\documentclass[journal]{IEEEtran}

\usepackage{color,xcolor}
\newcommand{\tr}[1]{\textcolor{violet}{#1}}

\usepackage{dsfont,amsmath,amsthm,algorithm,algorithmic}
\usepackage{graphicx,epstopdf,cases} %
\usepackage{hyperref} 
\usepackage{cleveref}
\usepackage{epstopdf}

\definecolor{darkblue}{rgb}{0.0, 0.0, 0.55}

\newtheorem{remark}{Remark}

\newcommand{\h}[1]{\mathbf{#1}}

\theoremstyle{definition}
\newtheorem{definition}{Definition}
\usepackage[utf8]{inputenc}
\usepackage[english]{babel}
\newtheorem{theorem}{Theorem}

\newtheorem{lemma}{Lemma}
\newtheorem{prop}{Proposition}

\ifCLASSINFOpdf
\else
\fi
\begin{document}
	%
	\title{Accelerated Schemes for the $L_1/L_2$ Minimization}
	%
	%
	%
	
	\author{Chao Wang, Ming Yan, Yaghoub Rahimi, Yifei Lou  
		\thanks{C. Wang and Y. Lou are with the Department of Mathematical Sciences, University of Texas at Dallas, Richardson, TX 75080 USA (E-mail:  chaowang.hk@gmail.com, yifei.lou@utdallas.edu). Y. Lou was partially supported by NSF Awards DMS 1522786 and 1846690. }
		\thanks{Y. Rahimi is with the School of Mathematics, Georgia Institute of Technology, Atlanta, GA 30332 USA (E-mail: yrahimi6@gatech.edu).}
		\thanks{M. Yan is with the Department of Computational Mathematics, Science and Engineering (CMSE) and the Department of Mathematics, Michigan State University, East Lansing, MI, 48824 USA  (Email: yanm@math.msu.edu). M. Yan was partially supported by NSF award DMS 1621798.  }}

	\maketitle
	
	\begin{abstract}
		In this paper, we consider the $L_1/L_2 $ minimization for  sparse recovery and study its relationship with the $L_1$-$ \alpha L_2 $ model. Based on this relationship, we propose three numerical algorithms to minimize this ratio model, two of which work as adaptive schemes and greatly reduce the computation time.  Focusing on the two adaptive schemes, we  discuss their connection  to existing approaches  and analyze their convergence.  The experimental results demonstrate that the proposed algorithms are comparable to state-of-the-art methods in sparse recovery and work particularly well when the ground-truth signal has a high dynamic range. Lastly, we reveal some empirical evidence on the exact $L_1$ recovery  under various combinations of sparsity, coherence, and dynamic ranges, which calls for theoretical justification in the future.
	\end{abstract}
	

	%
	\IEEEpeerreviewmaketitle

	\section{Introduction}
	\IEEEPARstart{I}{n} various science and engineering applications, one aims to seek for a low-dimensional representation from high-dimensional data, and 
	sparsity is a crucial assumption.  For example,  it is reasonable to assume in machine learning  \cite{tibshirani96lasso} that only a few features correspond to the response. In image processing \cite{ROF1992nonlinear}, the restored images are often piecewise constant, which means that gradients are sparse.  In  non-negative matrix factorization  \cite{plemmons1994nonnegative}, the low-rank decomposition  enforces sparsity with respect to singular values.

	Sparse signal recovery is to find the sparsest solution of $ A\mathbf{x} = \mathbf{b}$ where $ A \in \mathds{R}^{m \times n} $ ($ m \ll n$), $ \mathbf{x} \in \mathds{R}^n$, and $\mathbf{b} \in \mathds{R}^m$. We assume that $A$ has a full row rank  and $\h b$ is nonzero. 
	This problem is often referred to as \textit{compressed sensing} (CS) \cite{donoho06,CRT} in the sense that the sparse signal $\h x$ is compressible.
	Mathematically, it can be  formulated  by the $L_0$ minimization, 
	\begin{equation}\label{eq:l0}
	\min\limits_{\mathbf{x}\in \mathds{R}^n} \|\mathbf{x}\|_0 \quad \text{s.t.} \quad A\mathbf{x} = \mathbf{b}.
	\end{equation}
	Unfortunately,  the  $L_0$ problem  is known to be NP-hard \cite{natarajan95}. Various approaches  in sparse recovery have been investigated. Some greedy methods include orthogonal matching pursuit (OMP) \cite{OMP1993pati}, orthogonal least squares (OLS) \cite{OLS1989orthogonal}, and compressive sampling matching pursuit (CoSaMp) \cite{needell2009cosamp}. However, these greedy methods often lack of accuracy when $n$ is large. Alternatively,  approximations/relaxation approaches to  the $L_0$ norm have been sought. For example,
	convex relaxation, referred to as  \textit{basis pursuit} (BP) \cite{chenDS98}, replaces  $L_0$ in \eqref{eq:l0} with the $L_1$ norm. Recently, nonconvex models attract considerate amount of attentions due to their sharper approximations of $L_0$ compared to the $L_1$ norm. Some popular nonconvex models include $L_p$ \cite{chartrand07,xuCXZ12,laiXY13}, $L_1$-$L_2$ \cite{yinEX14,louYHX14},  transformed $L_1$ (TL1) \cite{lv2009unified,zhangX17,zhangX18}, nonnegative garrote \cite{breiman1995better}, and capped-$L_1$ \cite{peleg2008bilinear,zhang2009multi,shen2012likelihood}.
	Except for $L_1$-$L_2$, all of these nonconvex models involve one parameter to be determined and adjusted for different types of sparse recovery problems.

	In this paper, we study the ratio of $L_1 $ and $L_2 $ as a scale-invariant and parameter-free metric to approximate the desired scale-invariant $L_0$ norm. The ratio of $L_1 $ and $L_2$ can be traced back to \cite{hoyer2002}  as a sparsity measure,  and  its scale-invariant property was  explicitly mentioned in \cite{hurleyR09}. 
	Esser et al. \cite{esserLX13,yinEX14} focused on nonnegative signals and established the equivalence between  $L_1/L_2 $ and $L_0$.
	The ratio model was later formulated  as a nonlinear constraint that was solved by a lifted approach  \cite{esser2015lifted,esser2015resolving}.  Some applications of $L_1/L_2$ include blind deconvolution \cite{krishnan2011blind,repetti2015euclid} and  sparse filtering \cite{pham2017noise,jia2018sparse}.

	In our earlier work \cite{l1dl2}, we focused on a constrained minimization problem,
	\begin{equation}\label{equ:ratio}
	\min\limits_{\h x \in \mathds{R}^n}  \frac{\|\h x\|_1}{ \|\h x\|_2} \quad \mathrm{s.t.} \quad A \h x = \h b.
	\end{equation}
	Theoretically, we proved that  any $s$-sparse vector is a local minimizer of the $L_1 /L_2 $ model provided with a strong null space property (sNSP) condition. 
	Computationally, we considered to minimize \eqref{equ:ratio} via the alternating direction method of multipliers (ADMM) \cite{boydPCPE11admm}. In particular, we introduced two auxiliary variables and formed the augmented Lagrangian  as 
	\begin{align}
	L(\h x,\h y,\h z;\h v,\h w) = & \textstyle \frac{\|\h z\|_1}{ \|\h y\|_2} + I(A\h x-\h b)+  \frac{\rho_1}{2}\left\|\h x-\h y 
	+ \frac{1}{\rho_1} \h v\right\|_2^2 \nonumber \\
	& +  \textstyle  \frac{\rho_2}{2}\left\|\h x-\h z + \frac{1}{\rho_2} \h w \right\|_2^2,
	\end{align}
	where $I(\cdot)$ is  defined as 
	\begin{equation}\label{equ:indicator}
	I(\h t) = 
	\begin{cases}
	0,	&	\h t=\h 0,
	\\
	+\infty,	&	\text{otherwise}.
	\end{cases}
	\end{equation}
	There is a closed-form solution for each sub-problem. Please refer to \cite{l1dl2} for more details.

	This paper contributes  three  schemes to minimize \eqref{equ:ratio}. We demonstrate in experiments that the new schemes are computationally more efficiently compared to the previous ADMM approach.  The novelties of the paper are three-fold:
	\begin{enumerate}
		\item[(1)] Thanks to  the new schemes, $L_1/L_2$ can effectively deal with sparse signals with a high dynamic range, which is not the case for the  ADMM approach;
		\item[(2)] We reveal the connection of the proposed schemes to existing approaches, which helps to establish the convergence;
		\item[(3)] Our empirical results shed light about the effects of sparsity, coherence, and dynamic range on sparse recovery, which is new in the CS literature.
	\end{enumerate}

	The rest of the paper is organized as follows. \Cref{sect:model} is devoted to theoretical analysis on the relation between $L_1/L_2$  and $L_1$-$\alpha L_2$, which motivates three numerical schemes to minimize  $L_1/L_2$. We interpret the proposed schemes in line with some existing approaches in \Cref{sec:connection}, followed by convergence analysis  in \Cref{sec:convergence}.  We conduct extensive experiments in \Cref{sect:experiments} to demonstrate the performance of the $L_1/L_2 $ model  with three minimizing algorithms over  state-of-the-art methods in sparse recovery.
	\Cref{sect:discussion} presents how the classic $L_1$ approach behaves under different dynamic ranges and how sparsity, coherence, and dynamic range interplay on sparse recovery. 
	Finally, conclusions and future works are given in \Cref{sect:conclusion}.

	\section{Numerical schemes}\label{sect:model}
	
	We establish in \Cref{prop:relation} a link between the constrained $L_1/L_2 $ formulation \eqref{equ:ratio} and $L_1$-$\alpha L_2$, where $\alpha$ is a positive parameter. Immediately following this proposition, we develop a numerical algorithm for minimizing the ratio model. We further discuss two accelerated approaches in \Cref{sec:Adaptive}. 
	
	\begin{prop}\label{prop:relation}
		Denote 
		\begin{equation}\label{eq:l1dl2_alpha}
		\alpha^* := \inf\limits_{\h x \in \mathds{R}^n} \left\lbrace \frac{\|\h x\|_1}{\|\h x\|_2} \ \ \mathrm{ s.t. } \ A \h x =\h b \right\rbrace,
		\end{equation}
		and
		\begin{equation} \label{eq:l1dl2sub_dca}
		T(\alpha)  :=  \inf_{\h x \in \mathds{R}^n} \left\lbrace \|\h x\|_1 -\alpha\|\h x\|_2 \ \  \mathrm{ s.t. } \ A \h x = \h b \right\rbrace,
		\end{equation}
	then	we have 
		\begin{enumerate}
			\item[(a)] if $ T(\alpha) < 0$,  then $\alpha>\alpha^\ast$;
			\item[(b)] if $ T(\alpha) \geq 0$,  then $\alpha\leq\alpha^\ast$;
			\item[(c)] if $T(\alpha) =0$,  then $\alpha=\alpha^\ast.$
		\end{enumerate}	
	
	\begin{proof}
	Denote the feasible set of~\eqref{eq:l1dl2_alpha} by $\mathbf{F} = \{ \h  x \mid A \h x=\h b \} $. Since $\h b \neq 0 $ then $\h 0 \notin \mathbf{F}$. 
\begin{enumerate}
			\item[(a)] If $T(\alpha) < 0 $, then there exists $\h x \in \mathbf{F} $ such that $\|\h x\|_1 - \alpha \|\h x\|_2 < 0 $, which implies that $\alpha > \frac{\|\h x\|_1}{\|\h x\|_2} $. Therefore, we have $\alpha > \alpha^{*} $.
			\item[(b)] If $T(\alpha)  \geq 0$, then for all $\h x \in \mathbf{F} $ we have $\|\h x\|_1 - \alpha \|\h x\|_2 \geq 0 $. So $\alpha \leq \frac{\|\h x\|_1}{\|\h x\|_2} $ and hence $\alpha \leq \inf\limits_{x \in \mathbf{F}}  \frac{\|\h x\|_1}{\|\h x \|_2}  = \alpha^* $, i.e., $\alpha \leq \alpha^* $.
			\item[(c)] If  $T(\alpha)  =  0,$ then by part (b) we get $\alpha \leq \alpha^* $. Furthermore, there exists a sequence $\{ \h x_n\} \subset \mathbf{F} $ such that $\lim\limits_{n\to\infty}  \left( \|\h x_n\|_1-\alpha \|\h x_n\|_2 \right) = 0$.  Since $\h x_n \in \mathbf{F}$, we have $\|\h b\| = \|A\h x_n\| \leq \|A\|\|\h x_n\| $.  Hence, $\{\h x_n\}$ has a lower bounded, i.e.  $\|\h x_n\| \geq \|\h b\|/\|A\|$ for all $ n$, then we get
			$\lim\limits_{n\to\infty} \left( \|\h x_n\|_1-\alpha \|\h x_n\|_2 \right)/\|\h x_n\|_2 = 0$, which means
			$\alpha^* \leq \lim\limits_{n\rightarrow \infty}  \|\h x_n\|_1/\|\h x_n\|_2 = \alpha$. Therefore, we have $\alpha = \alpha^* $.
	\end{enumerate}
\end{proof}
	\end{prop}

	\subsection{Bisection Search}\label{sec:BS}
	
	It follows from \Cref{prop:relation} that the optimal value of $L_1/L_2$ equals to the value of $\alpha$ in the $L_1$-$\alpha L_2$ model if the objective value of $L_1$-$\alpha L_2$ is zero. That is to say, the optimal value of the ratio model is the root of $T(\alpha)$, which can be obtained by \textit{bisection search}. 
	Moreover, we have upper/lower bounds   of $\alpha$, i.e.,  $\alpha\in[1,\sqrt n]$, since $
		\| \h x \|_2 \leq \|\h x\|_1 \leq \sqrt{n} \|\h x\|_2, \ \forall \h x\in\mathds{R}^n$  \cite{golub1996matrix}.
	The procedure goes as follows: we start with an initial range of $\alpha$ to be $\lbrack 1, \ \sqrt{n} \rbrack$ and an initial value of $\alpha^{(0)}$ in between. Then using this $\alpha^{(0)}$,  we  solve for the $L_1$-$\alpha^{(0)} L_2$ minimization via the difference-of-convex algorithm (DCA) \cite{louOX15};  more details on the DCA implementation will be given in \Cref{sec:Adaptive}. Based on the objective value of $T(\alpha^{(0)})$, we update the range of $\alpha$. Specifically if $T(\alpha^{(0)}) = 0$, then we find the minimum ratio and the corresponding minimizer $\h x^\ast$ in the $L_1$-$L_2$ model is also the minimizer of the $L_1/L_2$ model. If $T(\alpha^{(0)
	})>0$, then we update the range as $\lbrack \alpha^{(0)}, \ \sqrt{n}\rbrack.$   If $T(\alpha^{(0)})<0$, then the minimum ratio is smaller than $\alpha^{(0)}$, so we can shorten the range from  $\lbrack 1, \ \sqrt{n} \rbrack$ to $\lbrack 1, \ \alpha^{(0)} \rbrack. $ We can further shorten the internal as $\left\lbrack 1, \ \frac{\|\h x^{(k+1)}\|_1}{\|\h x^{(k+1)}\|_2}  \right\rbrack,$ as the objective value of $L_1$-$\frac{\|\h x^{(k+1)}\|_1}{\|\h x^{(k+1)}\|_2} L_2$ would be less than or equal to zero in the next iteration.    
	After the range is updated, we choose $\alpha^{(1)}$ using the middle point of  two end points and iterate.

	We summarize the entire process as Algorithm~\ref{alg:l1dl2DCA}, in which the stopping criterion is that the error between two adjacent $\alpha$ values  is small enough. As the algorithmic scheme  follows directly from bisection search, we refer the algorithm  as $L_1/L_2$-BS or BS if the context is clear. 
	The convergence of BS can be obtained in the same way that the bisection method converges. However, due to the nonconvex nature of the $L_1$-$\alpha L_2$ minimization \eqref{eq:l1dl2sub_dca}, there is no guarantee to find its global minimizer and hence the solution to \eqref{eq:l1dl2_alpha} may  be suboptimal.

	\begin{algorithm}
		\caption{The $L_1/L_2$ minimization via bisection search ($L_1/L_2$-BS). }
		\label{alg:l1dl2DCA}
		\begin{algorithmic}[1]
			\STATE{Input: $A\in \mathds{R}^{m\times n}, \h b\in \mathds{R}^{m}$, kMax,  and  $\epsilon \in \mathds{R}$}
			\STATE{Initialize: $ {\h  x}^{(0)},\alpha^{(0)}$, $lb = 1 $, $ub = \sqrt{n}$  and $k = 0$}
			
			\WHILE{$k < $ kMax or $|\h \alpha^{(k)}-\h \alpha^{(k-1)}| > \epsilon$}
			\STATE{$\h x^{(k+1)} = \arg\min\limits_{\h x \in \mathds{R}^n} \left\lbrace \|\h x\|_1 -\alpha^{(k)}\|\h x\|_2 \, \  \mathrm{s.t.} \, A \h x =\h  b \right\rbrace  $}
			\IF{$ \|\h x^{(k+1)}\|_1 -\alpha^{(k)}\|\h x^{(k+1)}\|_2 < 0$}
			\STATE $ub = \frac{\|\h x^{(k+1)}\|_1}{\|\h x^{(k+1)}\|_2} $
			\ELSIF{$ \|\h x^{(k+1)}\|_1 -\alpha^{(k)}\|\h x^{(k+1)}\|_2 > 0$}
			\STATE $lb = \alpha^{(k)} $
			\ELSE
			\STATE break
			\ENDIF
			
			$\alpha^{(k+1)} = \frac{ub + lb}{2} $
			\STATE{$k = k+1$}
			\ENDWHILE
			\RETURN $\h x^{(k)}$ \end{algorithmic} 
	\end{algorithm}

	\subsection{Adaptive  Algorithms}\label{sec:Adaptive}
	
	The BS algorithm  is computationally expensive, considering that the $L_1$-$\alpha L_2$ minimization is conducted for multiple times. To speed up, we discuss two variants of $L_1/L_2$-BS by updating  the parameter $\alpha$ iteratively while minimizing $\|\h x\|_1 - \alpha \|\h x\|_2$. 
	
	Following  the DCA framework \cite{TA98,phamLe2005dc} to minimize $\|\h x\|_1 - \alpha \|\h x\|_2$, we consider the objective function as the difference of two convex functions, i.e.,
	$
	\min\limits_{\h x \in \mathds{R}^n} g(\h x) - h(\h x).
	$
	By linearizing the second term $h(\cdot)$, the DCA iterates as follows,
	\begin{equation}\label{eq:DCA}
	\h x^{(k+1)} = \arg\min\limits_{\mathbf{x}\in \mathds{R}^n} g(\h x) - \left\langle \h  x, \nabla h(\h x^{(k)}) \right\rangle.
	\end{equation} 
	Particularly for the $L_1$-$\alpha L_2$ model, we have
	\begin{equation}\label{equ:g}
	g(\h x) = \| \h x\|_1 + I(A\h x -b) \quad \text{ and } \quad h(\h x ) = \alpha \|\h x\|_2,
	\end{equation}
	thus leading to the DCA update as 
	\begin{equation}
	\label{equ:DCA}
	\h x^{(k+1)} = \arg\min\limits_{\mathbf{x}\in \mathds{R}^n} g(\h x) - \left\langle \h  x, \frac{\alpha\h x^{(k)}}{\|\h x^{(k)}\|_2} \right\rangle.	
	\end{equation}

	Now we consider to update $\alpha$ iteratively by the ratio of the current solution, leading to the following scheme,
	\begin{equation}\label{equ:a1}
	\begin{cases}
	\h x^{(k+1)}  = \arg\min\limits_{\h x} \left\{g(\h x)  - \left\langle \h x,  \frac{\alpha^{(k)}\h x^{(k)}}{\|\h x^{(k)}\|_2} \right\rangle   \right\},\\
	\alpha^{(k+1)}  = \|\h x^{(k+1)}\|_1/\|\h x^{(k+1)}\|_2,
	\end{cases}
	\end{equation} 
	where $g$ is defined in \eqref{equ:g}. 
	Notice that the $\h x$-subproblem in \eqref{equ:a1} is a linear programming (LP) problem, which unfortunately has no guarantee that the optimal solution exists (as the problem can be unbounded). 
	To increase the robustness of the algorithm, we further incorporate a quadratic term into the linear problem, i.e., 
	\begin{equation}\label{equ:a2}
	\begin{cases}
	\h x^{(k+1)}  = \arg\min\limits_{\h x}\left\{ g(\h x)  - \left\langle \h x,  \frac{\alpha^{(k)}\h x^{(k)}}{\|\h x^{(k)}\|_2} \right\rangle + \frac{\beta}{2}\|\h x -  \h x^{(k)}\|_2^2 \right\},\\
	\alpha^{(k+1)} = \|\h x^{(k+1)}\|_1/\|\h x^{(k+1)}\|_2.
	\end{cases}
	\end{equation}

	We denote these two adaptive methods \eqref{equ:a1} and \eqref{equ:a2} as $L_1/L_2$-A1 and $L_1/L_2$-A2,  respectively or A1 and A2 for short. Both algorithms are  summarized in \Cref{alg:adaptive}.

	For the $\h x$ subproblem of $L_1/L_2$-A1, we convert it into an LP problem. Assume that $ \h x = \h x^+ -\h  x^- $ where $\h  x^+ \geq \h 0 $ and $ \h x^- \geq \h 0.$ Denote $ \bar{\h x} = \begin{bmatrix}\h x^+\\ \h x^-\end{bmatrix},$ then $ A\h x = \h b $ becomes $ \bar{A} \bar{\h x} = \h b$ with $ \bar{A} = \begin{bmatrix} A & -A \end{bmatrix} $. 
	Therefore, the $\h x$-subproblem becomes
	\begin{equation}
	\min_{\bar{\h x} \geq \h 0} \mathbf{c}^{T}\bar{\h x} \quad s.t. \quad \bar{A} \bar{\h x} = \h b,
	\end{equation}
	where $\mathbf{c} = \left[\mathbf{1}+\frac{\alpha^{(k)}\h x^{(k)}}{\|\h x^{(k)}\|_2}; \mathbf{1}-\frac{\alpha^{(k)}\h x^{(k)}}{\|\h x^{(k)}\|_2}\right]$.
	We adopt the software Gurobi \cite{optimization2014inc} to solve this LP problem.

	The $\h x$ subproblem of $L_1/L_2$-A2 is a quadratic programming problem, which can be solved via ADMM. By introducing an auxiliary variable $\h y$, we have  the augmented Lagrangian,
	\begin{equation}
	\begin{split}L_\rho(\h x,\h y;\h u)  &= \textstyle \|\h y\|_1 +  I(A\h x-\h b)  - \left\langle \h x,  \frac{\alpha^{(k)}\h x^{(k)}}{\|\h x^{(k)}\|_2} \right\rangle \\  & \textstyle + \frac{\beta}{2}\|\h x -\h x^{(k)}\|_2^2 + \h u^T(\h x-\h y)+\frac \rho 2\|\h x-\h y\|_2^2. 
	\end{split}
	\end{equation}
	Then the ADMM iteration goes as follows
	\begin{equation}\label{eq:L1ADMM}
	\begin{cases}
	 \h x_{j+1}  = \arg\min\limits_{\h x}L_\rho(\h x,\h y_{j};\h u_{j}),\\
	\h y_{j+1} = 
	\arg\min\limits_{\h y}L_\rho(\h x_{j+1},\h y;\h u_{j}) ,\\
	\h u_{j+1} = \h u_{j}+\rho\left(\h x_{j+1}-\h y_{j+1}\right),
	\end{cases}
	\end{equation}
	where the subscript $j$ indexes the inner loop, as opposed to the superscript $k$ for outer iterations used in \eqref{equ:a2}. The $\h x$-subproblem of \eqref{eq:L1ADMM}  is a projection problem to minimize 
	$$\left\|\h x -\frac{\beta \h x^{(k)}- \h u_{j} + \rho \h y_{j}+\frac{\alpha^{(k)}\h x^{(k)}}{\|\h x^{(k)}\|_2}}{\beta +\rho}\right\|_2^2 ,
	$$ under the constraint of $A\h x=\h b $. Since the closed-form solution of projecting a vector $\h z$ to this constraint is 
	\begin{equation}
	\label{equ:projection}
	\mathbf{proj}(\h z ) = \h z - A^T(A A^T)^{-1}\h (A\h z -\h b),
	\end{equation}
 the $\h x$-update  is given by
	\[
	\h x_{j+1} = \mathbf{proj}\left(\frac{\beta \h x^{(k)}- \h u_{j}+ \rho \h y_{j}+\frac{\alpha^{(k)}\h x^{(k)}}{\|\h x^{(k)}\|_2}}{\beta +\rho}\right).
	\]
	The $\h y$-subproblem of \eqref{eq:L1ADMM} is equivalent to
	\begin{equation*}
	\h y_{j+1} =  \arg \min\limits_{\h y} \left\{\|\h y\|_1+\frac \rho 2 \left\|\h y-\h x_{j+1}-\frac {\h u_{j}} \rho\right\|_2^2\right\}.
	\end{equation*}
	It has a closed-form solution via \textit{soft shrinkage}, i.e.,
	\begin{equation}
	\textstyle \h y_{j+1} = \mathbf{shrink}\left(\h x_{j+1}+\frac {\h u_{j}} \rho, \frac 1 \rho\right),
	\end{equation} 
	with $ \mathbf{shrink}(\h v, \mu) = \mathrm{sign}(\h v)\max\left(|\h v|-\mu, 0\right).
$
	
	\begin{algorithm}[t]
		\caption{The $L_1/L_2$ minimization via adaptive selection method ($L_1/L_2$-A1 or A2). }
		\label{alg:adaptive}
		\begin{algorithmic}[1]
			\STATE{Input: $A\in \mathds{R}^{m\times n}, \h b\in \mathds{R}^{m}$, kMax, and $ \epsilon \in \mathds{R}$}
			\STATE{initialization: $ {\h  x}^{(0)}, \alpha^{(0)}$  and $k = 1$}
			\WHILE{$k < $ kMax or $\|\h x^{(k)}-\h x^{(k-1)}\|_2/\|\h x^{(k)}\| > \epsilon$}
			\STATE\label{line3}{
				\begin{equation*}
				\begin{cases}
				\text{Update } \{\h x^{(k+1)}, \alpha^{(k+1)} \}\text{ by } \eqref{equ:a1} & \text{ for A1}\\
				\text{Update } \{\h x^{(k+1)}, \alpha^{(k+1)}\} \text{ by } \eqref{equ:a2} & \text{ for A2}\\
				
				\end{cases}
				\end{equation*}
			}
			\STATE{$k = k+1$}
			\ENDWHILE
			\RETURN $\h x^{(k)}$ \end{algorithmic} 
	\end{algorithm}

	\section{Connections to previous works}\label{sec:connection}
	
	We try to interpret the proposed adaptive methods (A1 and A2)  in line with some existing approaches: parameter selection, generalized inverse power, and gradient-based methods. Our efforts  contribute to convergence analysis in \Cref{sec:convergence}. 
	
	\subsection{Parameter Selection}
	
	Recall that in $L_1/L_2$-BS, the ratio $L_1/L_2$ is minimized  when there exists a proper $\alpha^\ast$ such that $\|\h x^\ast\|_1 - \alpha^\ast \|\h x^\ast\|_2 = 0$  with $\h x^\ast = \arg \min\limits_{\h x } \left\{ \|\h x\|_1 - \alpha^\ast \|\h x\|_2 \ \mathrm{s.t.} \ A\h x = \h b \right\}$. We can regard this process as a root-finding problem for $\alpha^\ast$, which often occurs in parameter selection. For example, in the discrepancy principle method \cite{dp2012methods,youwei2012parameter,Idivergence2013minimization}, one aims to find a parameter $\alpha$ such that 
	the resulting data-fitting term is close to the noise level.
	In particular, we  represent this process by
	\begin{equation}\label{eqn:para_2loops}
	\begin{cases}
	\h x^{(k+1)} = \arg\min\limits_{\h x} f(\h x, \alpha^{(k)}),\\
	\h \alpha^{(k+1)}  =  l(\h x^{(k+1)}, \alpha^{(k)}),
	\end{cases}
	\end{equation} 
	where  $f(\cdot)$ is a general objective function to be minimized and $l(\cdot)$ is a certain scheme to update $\alpha$ so that discrepancy principle holds.	 
	Typically, an inner loop is required to find the solution of $\h x$-subproblem, followed by updating this parameter in an outer iteration.  We further present
	the $j$-th inner iteration at the $k$-th outer iteration by 
	\begin{equation}\label{eq:para_inner/outer}
	\h x_{j+1} = \Psi (\h x_j, \alpha^{(k)}),
	\end{equation}
	for the $\h x$-subproblem in \eqref{eqn:para_2loops}.
	
	To speed-up the process, Wen and Chan  \cite{youwei2012parameter} proposed an adaptive scheme that updates the parameter during the inner loop  such that it renders the current data-fitting term equal to the noise level. In other words, instead of updating $\alpha$ after minimizing $f$, they directly iterated 
	\begin{equation}\label{equ:para_scheme}
	\h x_{j+1} = \Psi(\h x_j, \alpha_{j+1}),
	\end{equation}
	in a way that $\{\h x_{j+1}, \alpha_{j+1}\}$ satisfies the discrepancy principle. 	In this way, only one loop is needed as opposed to inner/outer loops in \eqref{eq:para_inner/outer}. But it requires a closed-form solution for $\h x_{j+1} $ so one can perform a one-dimensional search for $\alpha_{j+1}$.

	The proposed BS scheme falls into the framework of \eqref{eqn:para_2loops} in that  the searching range of parameter is shorten every outer iteration. However, 
 $f$ in our BS method is the $L_1$-$\alpha L_2$ minimization that does not have a closed-form solution. 
As opposed to \eqref{equ:para_scheme}, we consider to update
\begin{equation}\label{equ:para_schem_ours}
		\h x_{j+1} = \Psi(\h x_j,\alpha_{j})
\end{equation} 
prior to updating $\alpha$. In other word, 
we update $\h x_{j+1}$ based on $\alpha_{j}$ rather than $\alpha_{j+1}$, the latter of which was adopted in the parameter-selection method \cite{youwei2012parameter}.
The rationale of \eqref{equ:para_schem_ours} is to guarantee that $\{\h x_{j+1}, \alpha_{j+1}\}$ satisfies $\|\h x_{j+1}\|_1- \alpha_{j+1}\|\h x_{j+1}\|_1 = 0$.  The iterative scheme \eqref{equ:para_schem_ours} is consistent with A1 or A2 (depending on the form of $\Psi$), if we change the notation from subscript $j$ to superscript $k$.

	\subsection{Generalized Inverse Power Methods}
	A standard technique to find the smallest eigenvalue of a  positive semi-definite symmetric matrix $B$ is the inverse power method \cite{golub1996matrix} that requires to iteratively solve 
	the  linear system,
	\begin{equation}
	\label{equ:invere_p}
	B\h x^{(k+1)}  =   \h x^{(k)}.
	\end{equation}
	The iteration converges to the smallest eigenvector of $B$, denoted by $\h x^\ast$. Then the smallest eigenvalue can be evaluated  by  $\lambda = q(\h x^\ast)$, where $q(\cdot)$ is
	Rayleigh quotient defined as $$q(\h x ) = \frac{\langle \h x, B \h x\rangle}{\|\h x\|_2^2}. $$
	Note that \eqref{equ:invere_p} is equivalent to the minimization problem
	\begin{equation}\label{equ:min_inverse_p}
	\h x^{(k+1)}=  \arg\min_{\h x }\left\{ \frac{1}{2}\langle \h x, B \h x\rangle -  \langle \h x^{(k)}, \h x\rangle \right\}. 
	\end{equation}

	It is well known in linear algebra \cite{golub1996matrix,trefethen1997numerical} that eigenvectors of $B$ are critical points of $\min\limits_{\h x} q(\h x)$ and the smallest eigenvalue/eigenvector can be found by \eqref{equ:min_inverse_p}. This idea is naturally extended to the nonlinear case in  \cite{IP2010inverse}, where a general quotient is considered,
	$
	q(\h x ) =  \frac{r(\h x)}{s(\h x)},
$
	with arbitrary functions $r(\cdot)$ and $s(\cdot)$. Similarly to \eqref{equ:min_inverse_p}, we have the corresponding scheme
	\begin{equation*}
	\h x^{(k+1)} =  \arg\min\limits_{\h x}\left\{ r(\h x) -  \langle \nabla s(\h x^{(k)}), \h x\rangle \right\}.
	\end{equation*}
 Following \cite{IP2010inverse}, we consider to update the eigenvalue $\lambda^{(k)} $ at each iteration to guarantee the algorithm's descent. In particular, the iterative scheme is  given by
		\begin{equation}\label{eqn:gIPM}
		\begin{cases}
		\h x^{(k+1)}  =  \arg\min\limits_{\h x }\left\{ r(\h x) -  \lambda^{(k)} \langle \nabla s(\h x^{(k)}), \h x\rangle \right\},\\
		\lambda^{(k+1)} =  \frac{r(\h x^{(k+1)})}{s(\h x^{(k+1)})}. 
		\end{cases}
		\end{equation}
If we choose $r(\h x) =  g(\h  x), \ s(\h x) = \|\h x\|_2, $ and denote $\lambda$ as $\alpha$, then the generalized inverse power method \eqref{eqn:gIPM} is $L_1/L_2$-A1.  
	In \cite{bresson2012convergence}, a modified inverse power method was proposed via the steepest descent flow. The iteration scheme is to incorporate a quadratic term in the objective function of the $\h x$-subproblem, which leads  to  $L_1/L_2$-A2. 

	\subsection{Gradient-based Methods}\label{sec:FBS}
	
	\begin{definition}	A \textit{critical point} of  a constrained optimization problem is a vector in the feasible set (satisfying the constraints) that is also a local maximum, minimum, or saddle point of the objective function.
	\end{definition}

	According to Karush-Kuhn-Tucker (KKT) conditions,  $\h x^\ast\neq \h 0$ is a critical point of \eqref{equ:ratio} if and only if there exists a vector $\h s$ such that 
	\begin{equation}
	\left\{\begin{array}{l}
	0  \in \frac{\partial \|\h x^\ast\|_1}{\|\h x^\ast\|_2} - \frac{\|\h x^\ast\|_1}{\|\h x^\ast\|_2^2} \frac{\h x^\ast }{\|\h x^\ast\|_2} + A^T \h s,	\\
	0  = A \h x^\ast - \h b .
	\end{array}\right.
	\end{equation}
	By introducing $
	\hat {\h s}  = \|\h x^\ast\|_2 \h \cdot \h s$,
	we have 
	\begin{equation}\label{equ:optimality}
	\left\{\begin{array}{l}
	0  \in \partial \|\h x^\ast\|_1- \frac{\|\h x^\ast\|_1}{\|\h x^\ast\|_2} \frac{\h x^\ast }{\|\h x^\ast\|_2} + A^T \hat{\h s}, 	\\
	0  = A \h x^\ast - \h b .
	\end{array}\right.
	\end{equation}	
	The  condition \eqref{equ:optimality} is also an optimality condition to another optimization problem:
	\begin{equation}\label{equ:reformulation}
	\min_{\h x} g(\h x) +w(\h x) ,
	\end{equation}
	where $g(\h x) $ is from \eqref{equ:g} and $w(\h x)$ is some function satisfying
	\begin{equation}\label{equ:F}
	\nabla w(\h x) = -\frac{\|\h x\|_1}{\|\h x\|_2} \frac{\h x}{\|\h x\|_2}. 
	\end{equation}
	Note that $w(\cdot)$ can not be explicitly determined from \eqref{equ:F}.

By applying a proximal gradient method (PGM) \cite{parikh2014proximal,fukushima1981generalized,combettes2011proximal} on the model \eqref{equ:reformulation}, we obtain the following  scheme
\begin{equation}\label{eq:G}
	\h x^{(k+1)} = \mathbf{prox}_{\frac{1}{\beta}g}\left(\h x^{(k)} - \frac{1}{\beta}\nabla w(\h x^{(k)}) \right),
\end{equation}
where $\mathbf{prox}_g(\h y ) = \arg\min\limits_{\h z}\left\{ g(\h z)+\frac{1}{2}\|\h z - \h y\|_2^2\right\}. $ This iterative scheme is 
the same as $L_1/L_2$-A2.

	As for $L_1/L_2$-A1, we can interpret it as a generalized conditional gradient method \cite{bredies2009generalized} that minimizes $g(\h x)+ w(\h x)$ by $\h x^{(k+1)}  = \min\limits_{\h y} \langle \nabla w(\h x^{(k)}), \h y\rangle + g(\h y). $

	\section{Convergence analysis}\label{sec:convergence}

	Following the discussion in \Cref{sec:FBS}, we present the convergence analysis. 
	We start with the convergence of A2, which is characterized in \Cref{thm:convergence}. To prove it, we need four lemmas, whose proofs are given in Appendix.	
	
		\begin{lemma}(Sufficient decreasing)\label{lema:decreasing}
		The sequence $\{\h x^{(k)},\alpha^{(k)}\}$ produced by $L_1/L_2$-A2 satisfies 
		$$  \alpha^{(k)}- \alpha^{(k+1)} \geq \frac{\beta}{\tr{2}\| \h x^{(k+1)}\|_2} \|\h x^{(k+1)} - \h x^{(k)}\|_2^2, \quad \forall k>0.$$
	\end{lemma}
	
	The next two lemmas (\Cref{lema_lipschitz} and \Cref{F_lipschitz}) discuss the Lipschitz properties. 

	\begin{lemma}\label{lema_lipschitz}
		Define $L = \frac{1}{\|A^T(A A^T)^{-1}\h b\|_2}$. Then for any $\h x, \h y \in \mathds{R}^n$ satisfying  $A \h x = A \h y = \h b$,  we have
		\begin{equation*}
		\left\|\frac{\h x}{\|\h x\|_2} -  \frac{\h y}{\|\h y\|_2}\right\|_2 \leq L\|\h x - \h y\|_2.
		\end{equation*}
	\end{lemma}

	Since the gradient of the $L_2$ norm is $\nabla\|\h x \|_2  = \frac{\h x}{\|\h x\|_2}$, \Cref{lema_lipschitz} implies that
	the gradient of  Euclidean norm is Lipschitz-continuous in the domain $\{ \h x \ | \ A\h x = \h b\}$. 
	The next lemma is about the Lipschitz property for the implicit function $w(\cdot)$  that satisfies \eqref{equ:F}.
	
	\begin{lemma}\label{F_lipschitz}
Given	$L$ defined in \Cref{lema_lipschitz}.	For any $\h x, \h y \in \mathds{R}^n$ satisfying  $A \h x = A \h y = \h b$,  then
		\begin{equation}\label{eq:lemma3}
		\left\|\nabla w(\h x)- \nabla w(\h y)\right\|_2 \leq L_w\|\h x - \h y\|_2,
		\end{equation}
	for $w$ satisfying \eqref{equ:F} and $L_w = 2\sqrt{n}L$.
	\end{lemma}

	\begin{lemma}\label{lemma:G}
		Given $g(\cdot)$ defined in \eqref{equ:g} and suppose $w(\cdot)$ satisfies \eqref{equ:F}, we denote 
		\begin{equation}\label{equ:G}
		\textstyle \Phi(\h x ) := \beta \left(\h x - \mathrm{prox}_{\frac{1}{\beta}g}\Big(\h x -\frac 1 {\beta}\nabla w(\h x) \Big)\right),
		\end{equation}
		for an arbitrary $\beta>0$.	Then we have 
		\begin{enumerate}
			\item[(a)] $\Phi(\h x^\ast) = \h 0 $ if and only if $\h x^\ast$ is a critical point of \eqref{equ:ratio}; 
			\item[(b)] $\left\|\Phi(\h x)-\Phi(\h y) \right\|_2 \leq L_\Phi \|\h x- \h y\|_2$ with $L_\Phi = L_w+2\beta, $ for any $\h x, \h y \in \mathds{R}^n$ satisfying  $A \h x = A \h y = \h b$.
		\end{enumerate}
	\end{lemma}

It is stated in \eqref{eq:G} that $L_1/L_2$-A2 can be expressed as	 $\h x^{(k+1)} = \mathrm{prox}_{\frac{1}{\beta}g}\left(\h x^{(k)} -\frac 1 {\beta} \nabla w(\h x^{(k)}) \right)$. By the definition of $\Phi(\cdot)$ in \eqref{equ:G} and the decreasing property of $\|\h x\|_1/\|\h x\|_2$ in \Cref{lema:decreasing}, we can interpret A2 as a gradient descent method 
	$$\h x^{(k+1)}=\h x^{(k)}-\frac{1}{\beta}\Phi(\h x^{(k)}). $$
In the following theorem, we rely on \Cref{lemma:G} to show that the descent direction along  $\Phi(\cdot)$ leads to convergence.   

	\begin{theorem}\label{thm:convergence}
		Given a  sequence  $\{\h x^{(k)}, \alpha^{(k)}\}$ generated by $L_1/L_2$-A2. If $\{\h x^{(k)}\}$ is bounded, there exists a  subsequence that converges to a critical point of the ratio model \eqref{equ:ratio}. 
	\end{theorem}
	\begin{proof}
		According to \Cref{lema:decreasing},
		we know that $\alpha^{(k)}$ is decreasing and bounded from below, so there exists a scalar $\alpha^\ast$ such that $\alpha^{(k)} \rightarrow \alpha^{\ast}$. With the boundedness  assumption of $\h x$, we  get $\| \h x^{(k+1)}-\h x^{(k)} \|_2 \rightarrow 0$ from \Cref{lema:decreasing}, which implies that $\|\Phi(\h x^{(k)})\|_2 \rightarrow 0$. 
	The boundedness of $\h x^{(k)}$ also leads to   
		a convergent subsequence, i.e., $\h x^{(k_i)} \rightarrow \h x^{\ast}. $ Therefore, we have
		\begin{equation*}
		\begin{split}
			\|\Phi(\h x^\ast) \|_2  = & \left\|\Phi(\h x^\ast) - \Phi\left(\h x^{(k_i)}\right) + \Phi\left(\h x^{(k_i)}\right)\right\|_2 \\
		\leq & \left\|\Phi(\h x^\ast) - \Phi\left(\h x^{(k_i)}\right)\right\|_2 + \left\|\Phi\left(\h x^{(k_i)}\right)\right\|_2\\
		\leq& L_\Phi \| \h x^{(k_i)} - \h x^\ast\|_2 + \left\|\Phi\left(\h x^{(k_i)}\right)\right\|_2.
		\end{split}
		\end{equation*}
		As $k_i \rightarrow \infty$, we get $\|\Phi(\h x^\ast) \|_2 = 0$ and hence $\Phi(\h x^\ast) = \h 0. $ By \Cref{lemma:G},  $\{\h x^{(k_i)}\}$ converges to a critical point. 
	\end{proof}

	\begin{remark}
		 \Cref{thm:convergence} does not require that the step-size $\frac{1}{\beta}$ is small, which is typically for gradient-based methods. In our numerical tests, 
		   we can choose small $\beta$  and get good results.
	\end{remark}

	\begin{theorem}\label{thm:convergence_A1}
		Given a  sequence  $\{\h x^{(k)}, \alpha^{(k)}\}$ generated by $L_1/L_2$-A1. If $\{\h x^{(k)}\}$ is bounded, it has a convergent subsequence.
	\end{theorem}
	\begin{proof}
		Denote 
		$$z(\h x,\h x^{(k)}) := \|\h x\|_1 - \left\langle \h x,  \frac{\alpha^{(k)}\h x^{(k)}}{\|\h x^{(k)}\|_2} \right\rangle.$$ 
		Since $z(\h x^{(k)},\h x^{(k)}) = 0$ by the definition of $\alpha^{(k)}$, the minimal value of $z(\h x,\h x^{(k)})$ subject to  the constraint $\{ \h x \ | \ A\h x=\h b\} $ is less than or equal to zero.   Specifically, $z(\h x^{(k+1)}, \h x^{(k)})\leq 0.$
		As a result,  by Cauchy-Schwarz inequality,  we have
		\begin{equation}
		\label{equ:alpha_d}		
		\| \h x^{(k+1)}\|_1  \leq\left\langle \h x^{(k+1)},  \frac{\alpha^{(k)}\h x^{(k)}}{\|\h x^{(k)}\|_2} \right\rangle \leq \alpha^{(k)} \|\h x^{(k+1)}\|_2,
		\end{equation}
		which implies $\alpha^{(k+1)}\leq \alpha^{(k)}$. Since $\alpha^{(k)}\in [1, \sqrt n]$, the decreasing sequence of $\alpha^{(k)}$  converges, i.e., $\alpha^{(k)} \rightarrow \alpha^\ast$.  
		By the boundedness of $\h x^{(k)} $, it has a convergent subsequence, i.e, there exists a vector $\h x^\ast$ such that  $\h x^{(k_i)}\rightarrow\h x^{\ast}$. 
%
%
	\end{proof}		
	
	\begin{remark}
 The sufficient decrease property (\Cref{lema:decreasing}) does not hold for $\beta=0$ when	$L_1/L_2$-A2 reduces to A1. So, we cannot show that A1  converges to a critical point.	
\end{remark} 
 
 \begin{remark}
 According to \Cref{thm:convergence} and  \Cref{thm:convergence_A1},  we prove  that  either both algorithms diverge due to unboundedness or there exists a convergent subsequence. It is possible that the solution can be unbounded. For example, $A$ has a zero-column, then the corresponding entry can take $+\infty$ so that the ratio of $L_1$ and $L_2$ is minimized.
 In the numerical tests, we demonstrate  empirically that $\{\h x^{(k)} \}$ is always bounded and hence convergent for general (random) matrices $A$.
\end{remark}
	
	\section{Numerical experiments}\label{sect:experiments}
	
	In this section, we 
	compare the proposed algorithms with state-of-the-art methods in sparse recovery. 
	All the numerical experiments are conducted on a desktop with CPU (Intel i7-6700, 3.4GHz) and $\mathrm{MATLAB \ 9.2 \ (R2017a)}. $
	
	We focus on the sparse recovery problem with highly coherent matrices, on which standard $L_1$ models fail. Following the works of
	\cite{louYHX14,yinLHX14,DCT2012coherence}, we consider an oversampled discrete cosine transform (DCT), defined  as $A= [\h a_1, \h a_2, \cdots, \h a_n]\in \mathds{R}^{m\times n}$ with 
	\begin{equation}\label{eq:oversampledDCT}
	\textstyle \h a_j := \frac{1}{\sqrt{m}}\cos \left(\frac{2\pi \h w j}{F} \right), \quad j = 1, \cdots, n,
	\end{equation}
	where $\h w$ is a random vector that is uniformly distributed in $[0, 1]^m$ and $F\in \mathds{R}$ is a positive parameter to control the coherence in a way that a larger $F$ yields a more coherent matrix. Throughout the experiments, we consider  over-sampled DCT matrices of size $64\times 1024$. The ground truth $\h x\in \mathds{R}^{n}$ is simulated as an $s$-sparse signal, where  $s$ is the number of nonzero entries. As suggested in  \cite{DCT2012coherence}, we require a minimum separation at least $2F$ in the support of $\h x$. As for the values of non-zero elements, we follow the work of \cite{lorenz2011constructing} to consider sparse signals with a high dynamic range. Define the dynamic range of a signal $\h x$  as $
	\Theta(\h x) = \frac{\max\{|\h x_s|\}}{\min \{ |\h x_s|\}}, 
	$ which can be controlled by an exponential factor $D$. In particular, we simulate  $\h x_s$ by the following MATLAB command,
	$$
	\verb|xs = sign(randn(s,1)).*10.^(D*rand(s,1))|
	$$
	In the experiments,  we set  $D = 3$ and $5$, corresponding to $\Theta \approx 10^3$ and $10^5$, respectively.  Note that \verb|randn| and \verb|rand| are the $\mathrm{MATLAB}$ commands for the Gaussian distribution $\mathcal{N}(0,1)$ and the uniform distribution $\mathcal{U}(0, 1)$, respectively. To compare with our previous work \cite{l1dl2} of the   $L_1/L_2$ minimization, we also consider that the nonzero elements follow the Gaussian distribution, i.e., $ (\h x_s)_i \sim\mathcal{N}(0,1),  i = 1, 2, \cdots, s.$  

	The fidelity of sparse signal recovery is assessed in terms of \textit{success rate}, defined as the number of successful trials over the total number of trials. When the relative error between the ground truth $\h x$ and the reconstructed solution $\h x^\ast$, i.e., $\frac{\|\h x^\ast-  \h x\|_2}{\|\h x\|_2},$ is less than $10^{-3}$, we declare it as a \textit{success}. Moreover, we categorize the failure of not recovering the ground-truth signal as \textit{model/algorithm failures} and by comparing  the objective function $f(\cdot)$ at the ground truth $\h{x}$ and at the restored solution $\h{x}^\ast$. If $f(\h{x})>f(\h{x}^\ast)$, then $\h{x}$ is not a global minimizer of the model, so we regard it as a \textit{model failure}. If $f(\h{x})<f(\h x^\ast)$, then the algorithm does not reach a global minimizer. It is referred to as an \textit{algorithm failure}. Similarly to success rates, we can define model-failure rates and algorithm-failure rates.  

	\subsection{Algorithmic Comparison}\label{sect:alg}

	We present various computational aspects of the proposed algorithms, i.e.,  BS, A1, and A2, together with comparison to our previous ADMM approach \cite{l1dl2}.
	First of all, we attempt to demonstrate the convergence of all proposed algorithms using an example of $s=15$, $F = 15$ (so the minimal separation is 30), and nonzero elements following Gaussian distribution.  Since the ratio model is solved via the $L_1$-$\alpha L_2$ model, we plot the values of $\|\h x^{(k)} \|_1 -\alpha^{(k-1)} \|\h x^{(k)}\|_2$ and $\alpha^{(k)}$  versus iteration counter $k$ in \Cref{fig:obj}. For $L_1/L_2$-BS,  we record the value at each outer iteration and the stopping conditions are either the maximum outer iteration reaches 10 or $|\alpha^{(k)}-\alpha^{(k-1)}|\leq 10^{-2}$. For each iteration of A1, A2, and the inner loop of BS, the stopping criterions are the relative error $\| \h x^{(k)}-\h x^{(k-1)}\|_2 / \|\h x^{(k)}\|_2 \leq 10^{-8}.$  The left plot in \Cref{fig:obj} illustrates the convergence of  the three algorithms in the sense that $\|\h x^{(k)} \|_1 -\alpha^{(k)} \|\h x^{(k)}\|_2$ goes down. Both A1 and A2 are faster than BS as BS starts with a larger range of $\alpha$ as $[1,\sqrt{n}]=[1,32]$, while A1 and A2 start with a good initial value of $\alpha^{(0)}=\frac{\|\h x^{(0)}\|_1}{\|\h x^{(0)}\|_2}$, which is very close to the final optimal value $\alpha^\ast$.
	The right plot in \Cref{fig:obj} examines the evolution of $\alpha^{(k)}$, which gradually becomes stable and approaches  to a similar value around 3.06 for all three algorithms.
	\Cref{fig:obj} confirms the decrease property of $\alpha^{(k)}$ proved in \Cref{lema:decreasing}. 

	\begin{figure}[t]
		\centering 
		\begin{tabular}{cc}
			\includegraphics[width=0.2\textwidth]{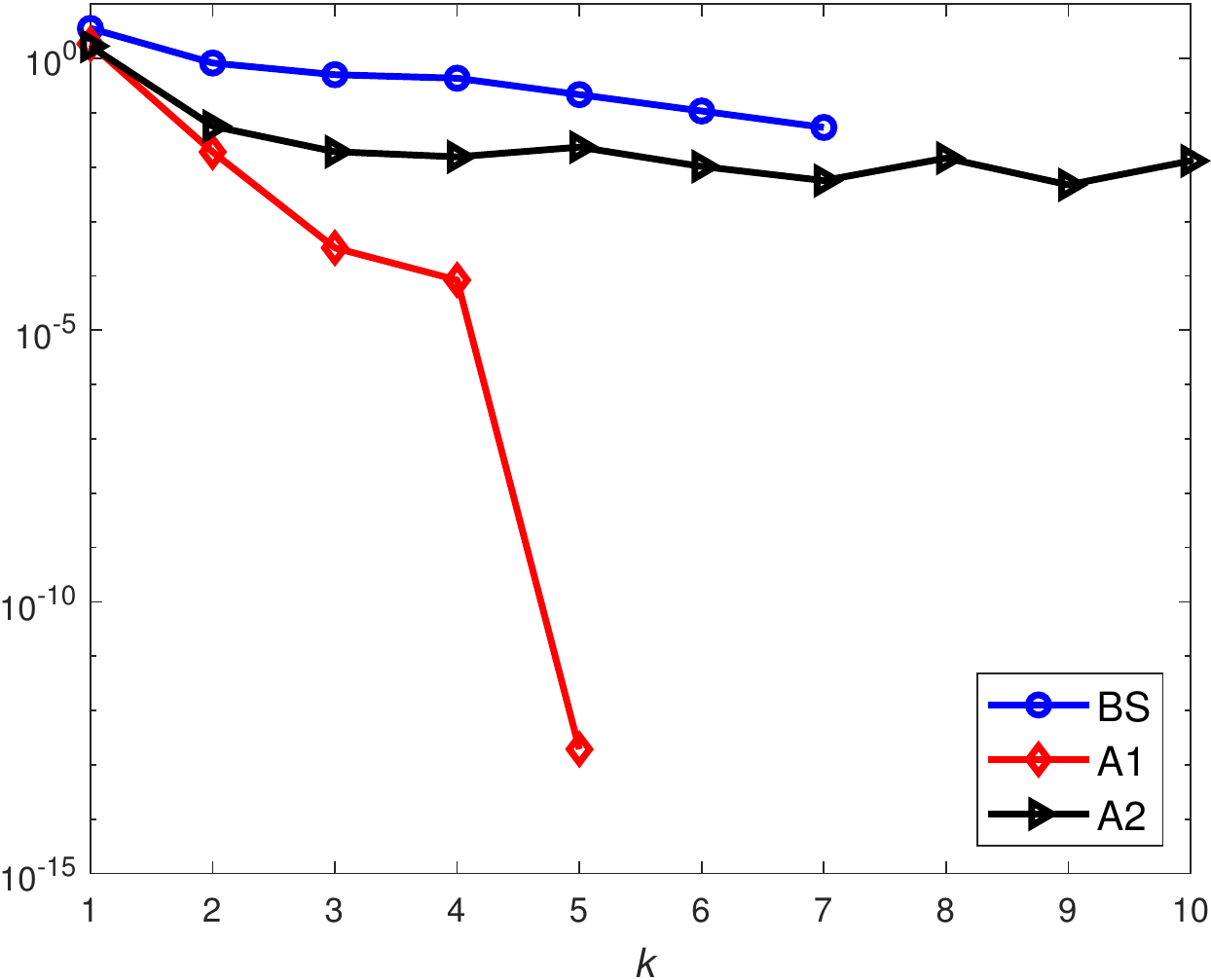} &
			\includegraphics[width=0.2\textwidth]{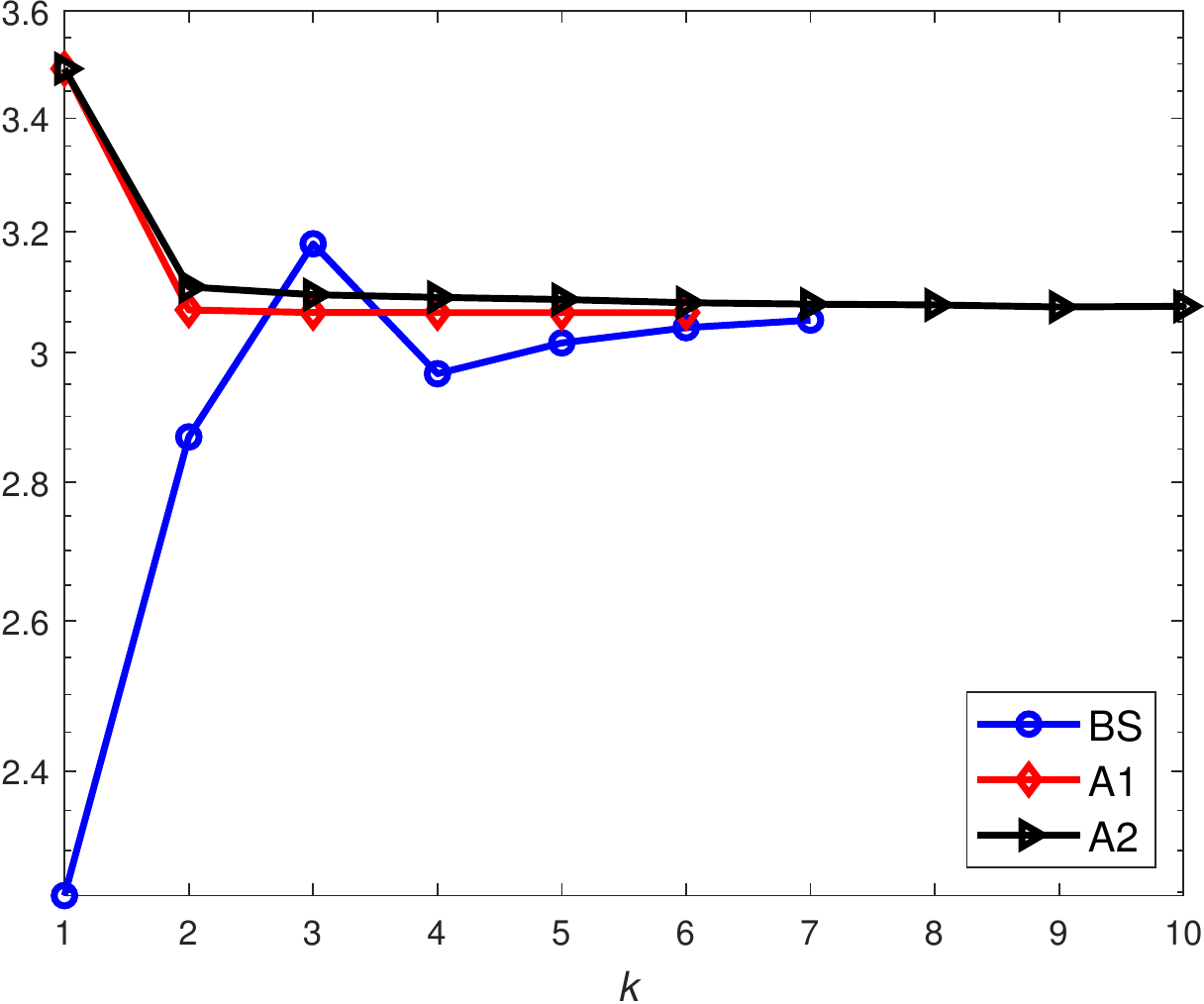}
		\end{tabular}
		\caption{Empirical analysis on convergence: $\|\h x^{(k)} \|_1 -\alpha^{(k-1)} \|\h x^{(k)}\|_2$ (left) and $\alpha^{(k)}$ (right) versus iteration counter $k$ for BS, A1, and A2. }
		\label{fig:obj}
	\end{figure}

	In \Cref{thm:convergence}, we require the sequence $\{\h x^{(k)}\}$ to be bounded for the convergence analysis. Here we aim at an empirical verification on the boundedness. In particular, 
	we test on various kinds of linear systems with $F \in\{1,20\}$ and sparsity ranging from 2 to 22. In each setting, we randomly generate 50 pairs of  ground-truth signals and linear systems to  compute  the  $L_2$ norm of solutions obtained by A1 and A2, along with the  $L_2$ norm of ground-truth signals. The mean values of these $L_2$ norms are plotted in \Cref{fig:l2norm}.   As the maximum values are finite numbers, it means that the reconstructed signal is always bounded. 
	\Cref{fig:l2norm} also shows that the $L_2$ norms of A1 and A2 align quite well with the ground truth when the sparsity is below 14, no matter the system is coherent or not. When the matrix is highly coherent with more nonzero elements, both A1 and A2 give much larger values of the $L_2$ norm compared to the ground truth. It is  because a larger $L_2$ norm gives rise to a smaller value in the ratio of $L_1/L_2$ that we try to minimize. In any cases, the solutions of both A1 and A2 are shown to be bounded.  

	\begin{figure}[t]
		\centering 
		\begin{tabular}{cc}
			$F=1$ & $F=20$  \\
			\includegraphics[width=0.2\textwidth]{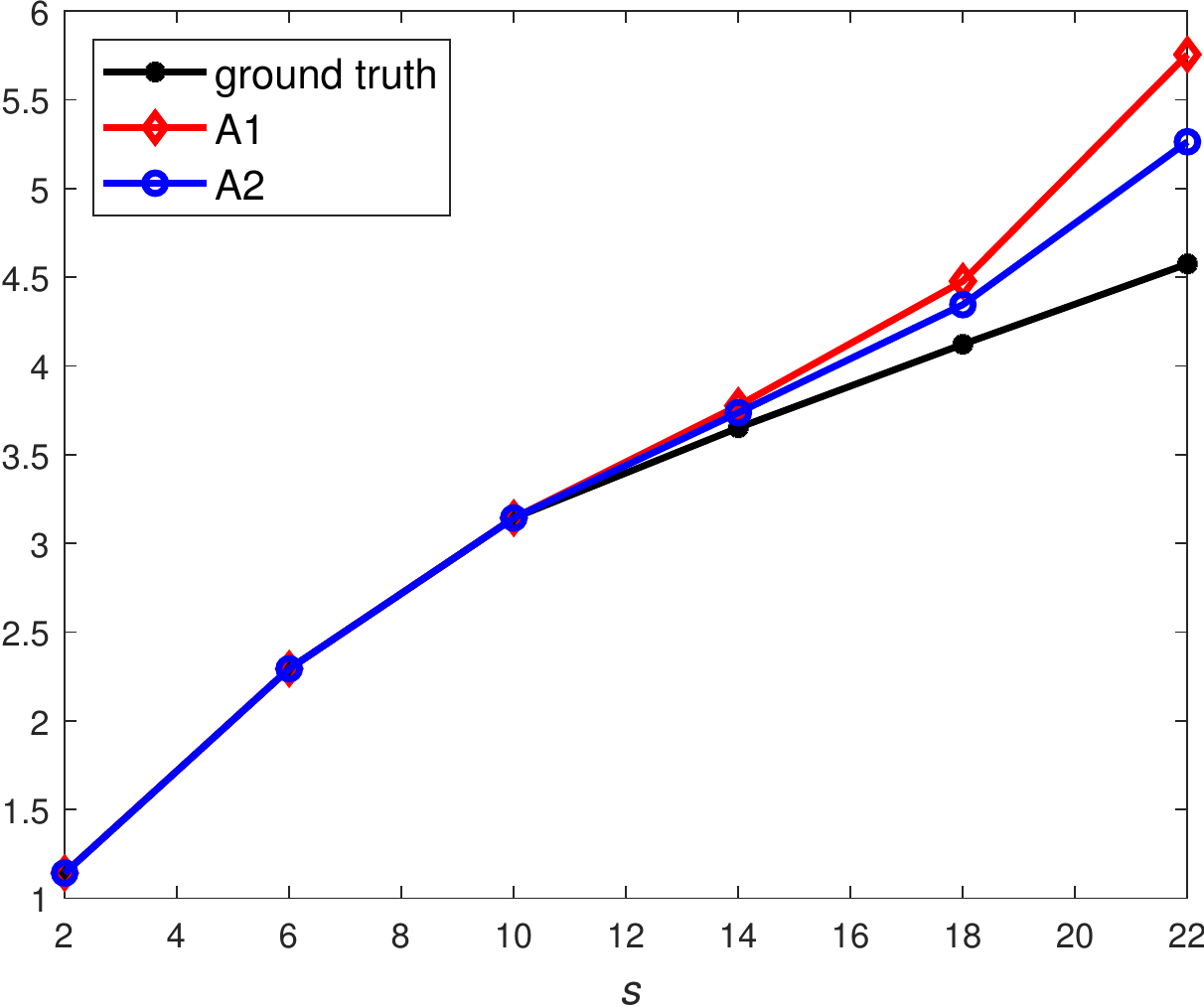}&
			\includegraphics[width=0.2\textwidth]{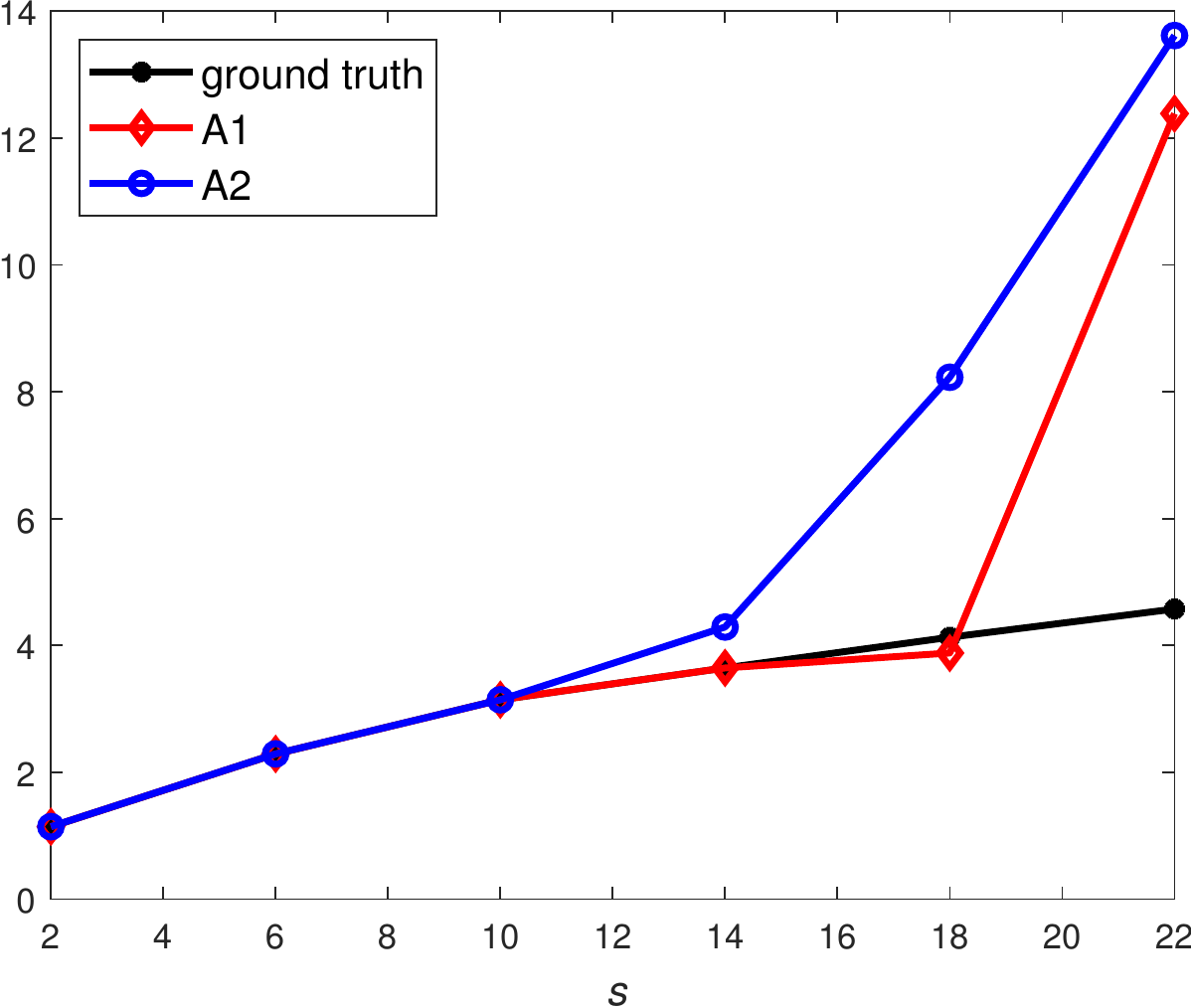}\\
		\end{tabular}
		\caption{The $L_2$ norm of the ground truth vectors as well as  the reconstructed solutions by A1 and A2.  }
		\label{fig:l2norm}
	\end{figure}

	Next, we compare the three algorithms with our previous ADMM approach \cite{l1dl2}. We consider $F = 1$ and $20$ with nonzero elements following the Gaussian distribution or having  high dynamic ranges. 
	We randomly simulate 50 trials for each sparsity level and compute the average of success rates, algorithm-failure rates, and computation time. 
	The Gaussian case is illustrated in \Cref{fig:alg_SR_time_G},
	showing that ADMM is the worst in terms of success rates partly due to high algorithm failure rates. Here, $\rho_1 = \rho_2 = 2000$ for ADMM and $\beta = 1, \rho = 20$ for A2.  In addition, BS achieves the highest success rates but is the slowest. Both A1 and A2 have similar performance to  BS with much reduced computation time. 
	\Cref{fig:alg_SR_time_D35} examines the case of the dynamic range for the non-zero values in $\h x$ with $D=3$ and $5$. Here we set $\beta = 10^{-5}$ and $\rho = 0.3$ for A2,  while $\rho_1 = \rho_2 = 100$ for ADMM. 
	Similar performance is observed as the Gaussian case. 
	In summary, we rate A1 as the most efficient algorithm for minimizing the ratio model with a balanced performance between accuracy and computational costs. We also observe that 
	all the algorithms tend to give better performance in terms of 
	success rates with
	higher dynamic ranges, which seems counter-intuitive. We will revisit this phenomenon in \Cref{sect:discussion}.
	
	\begin{figure*}[tbhp]
		\centering 
		\begin{tabular}{ccc}
			success rates & algorithm-failure rates & computation time  \\
			\includegraphics[width=0.27\textwidth]{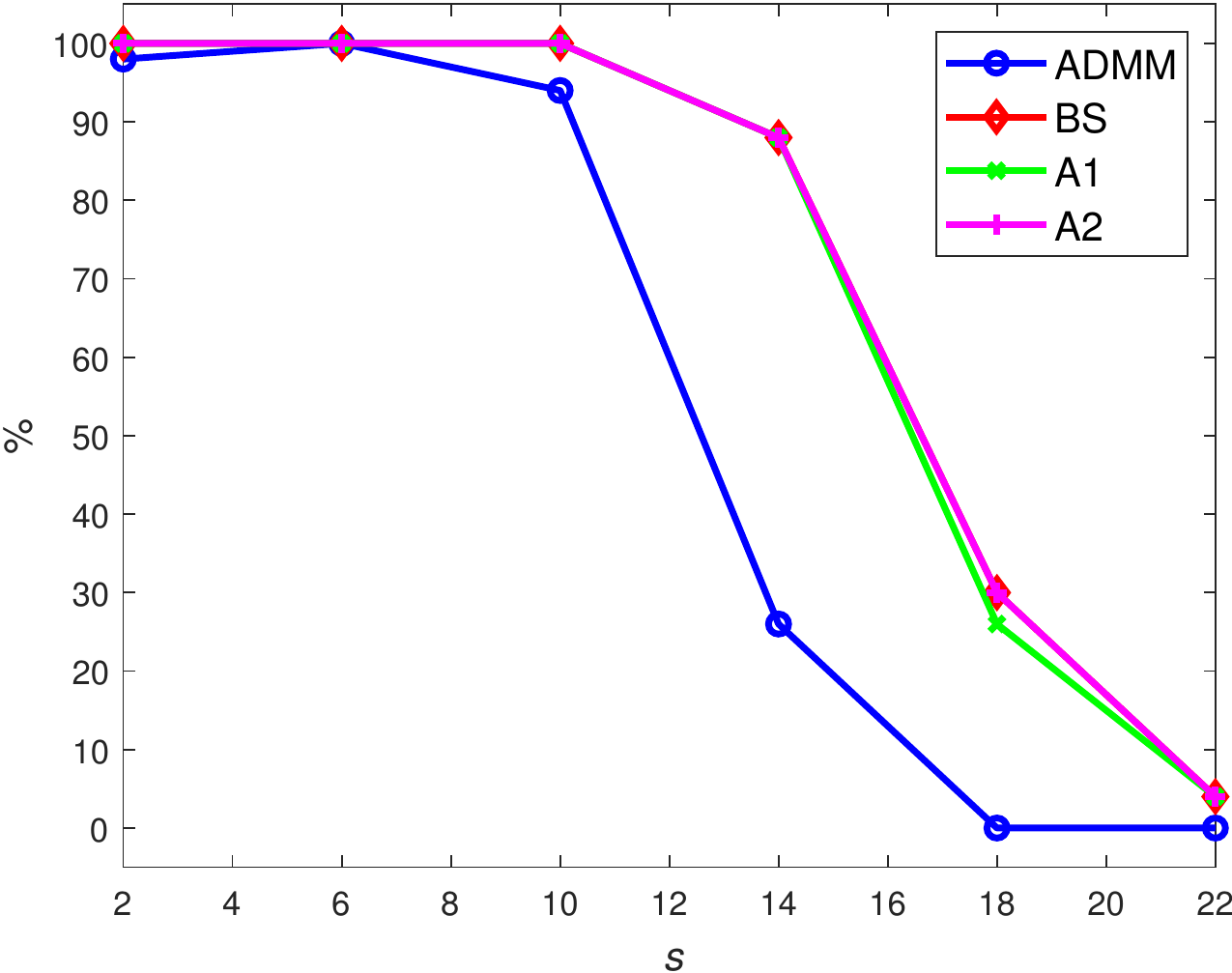}&
			\includegraphics[width=0.27\textwidth]{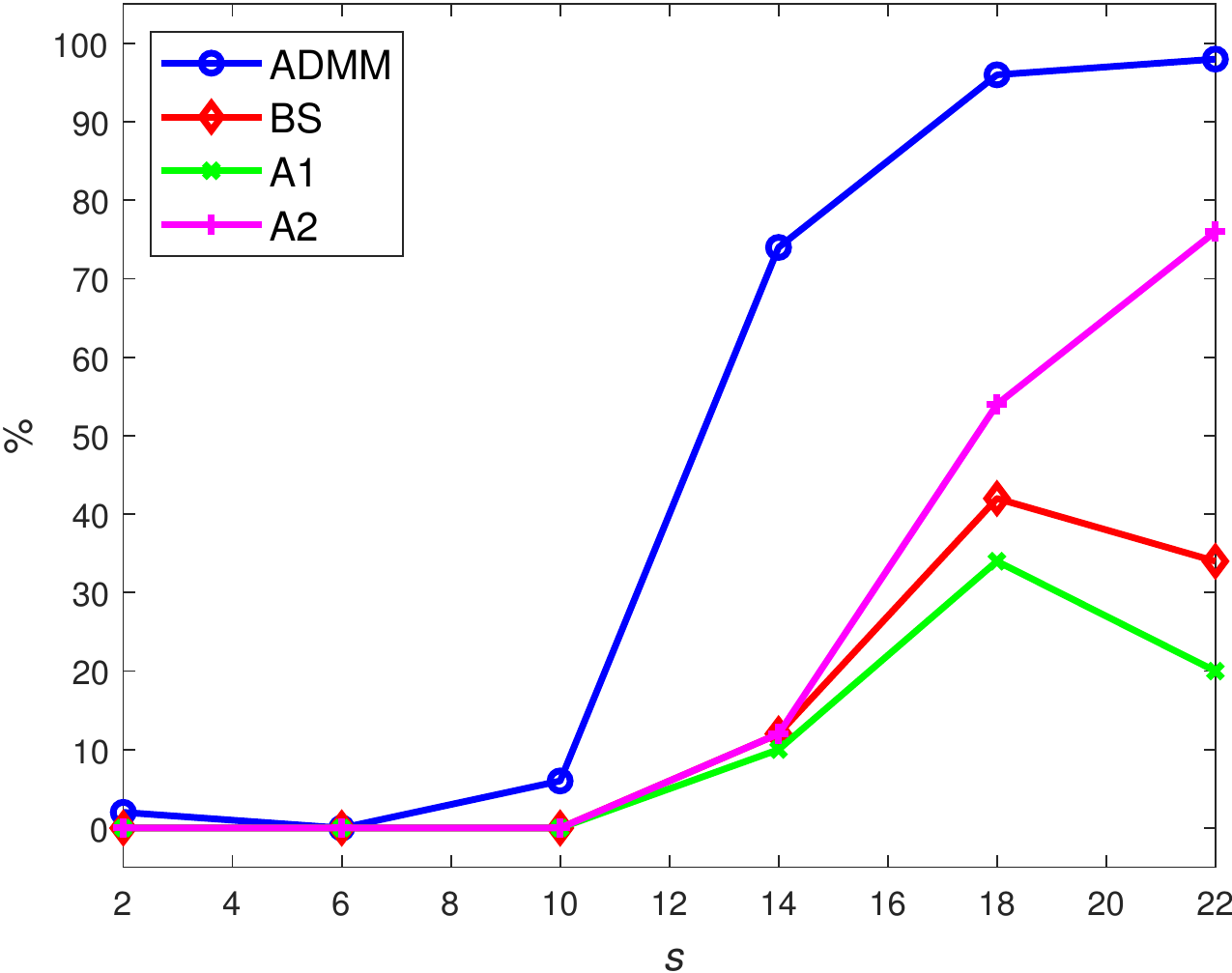}&
			\includegraphics[width=0.27\textwidth]{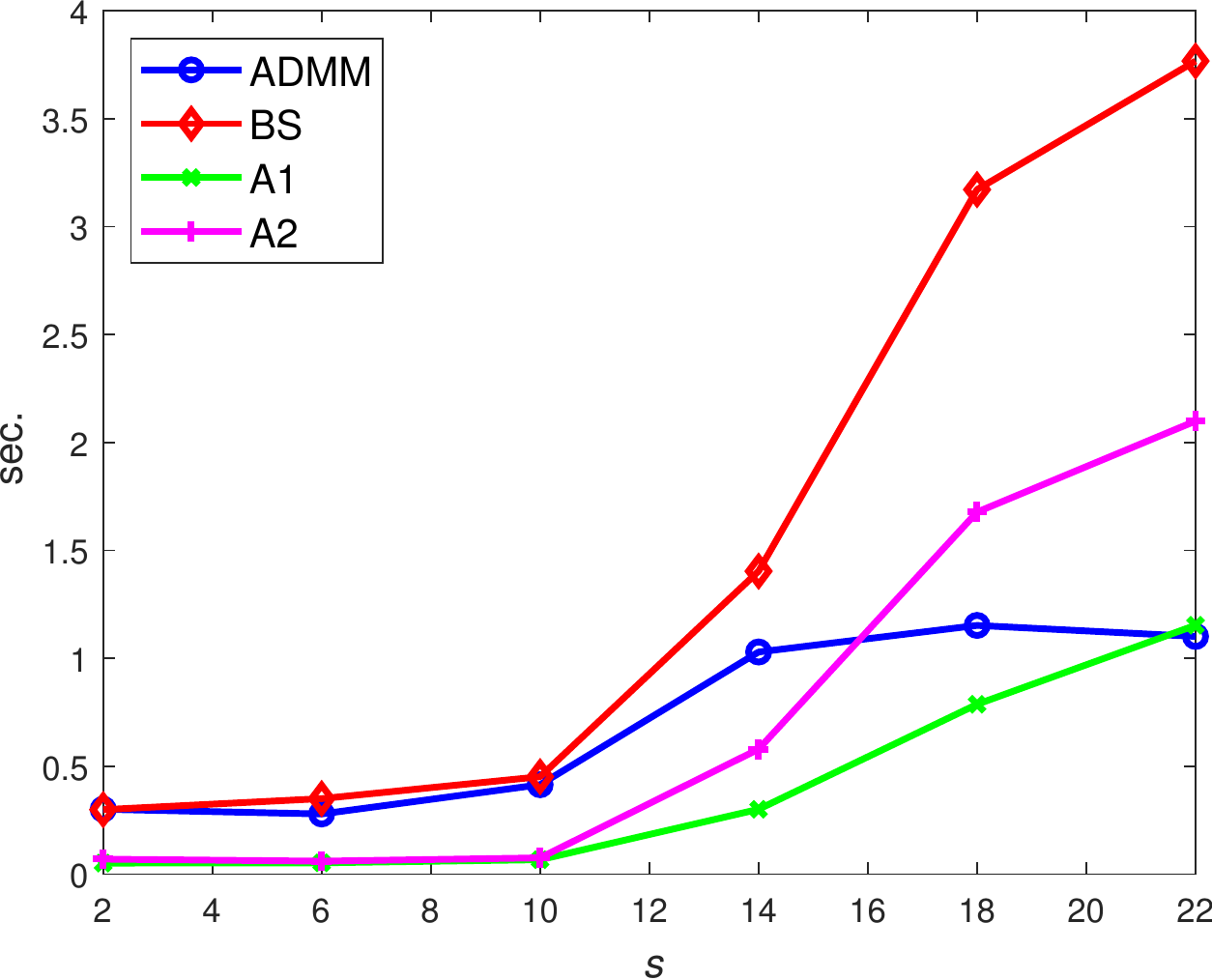} \\
			\includegraphics[width=0.27\textwidth]{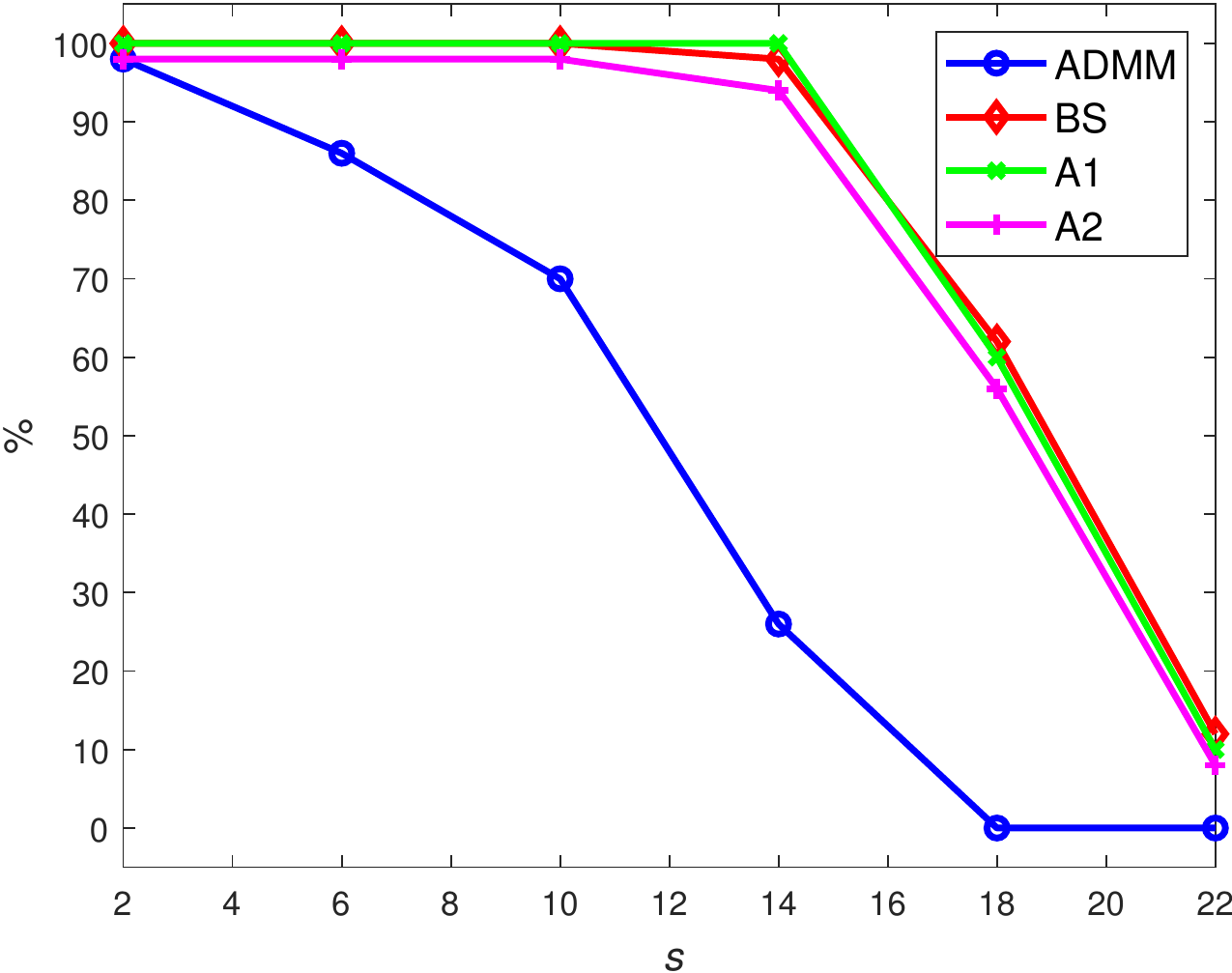}&
			\includegraphics[width=0.27\textwidth]{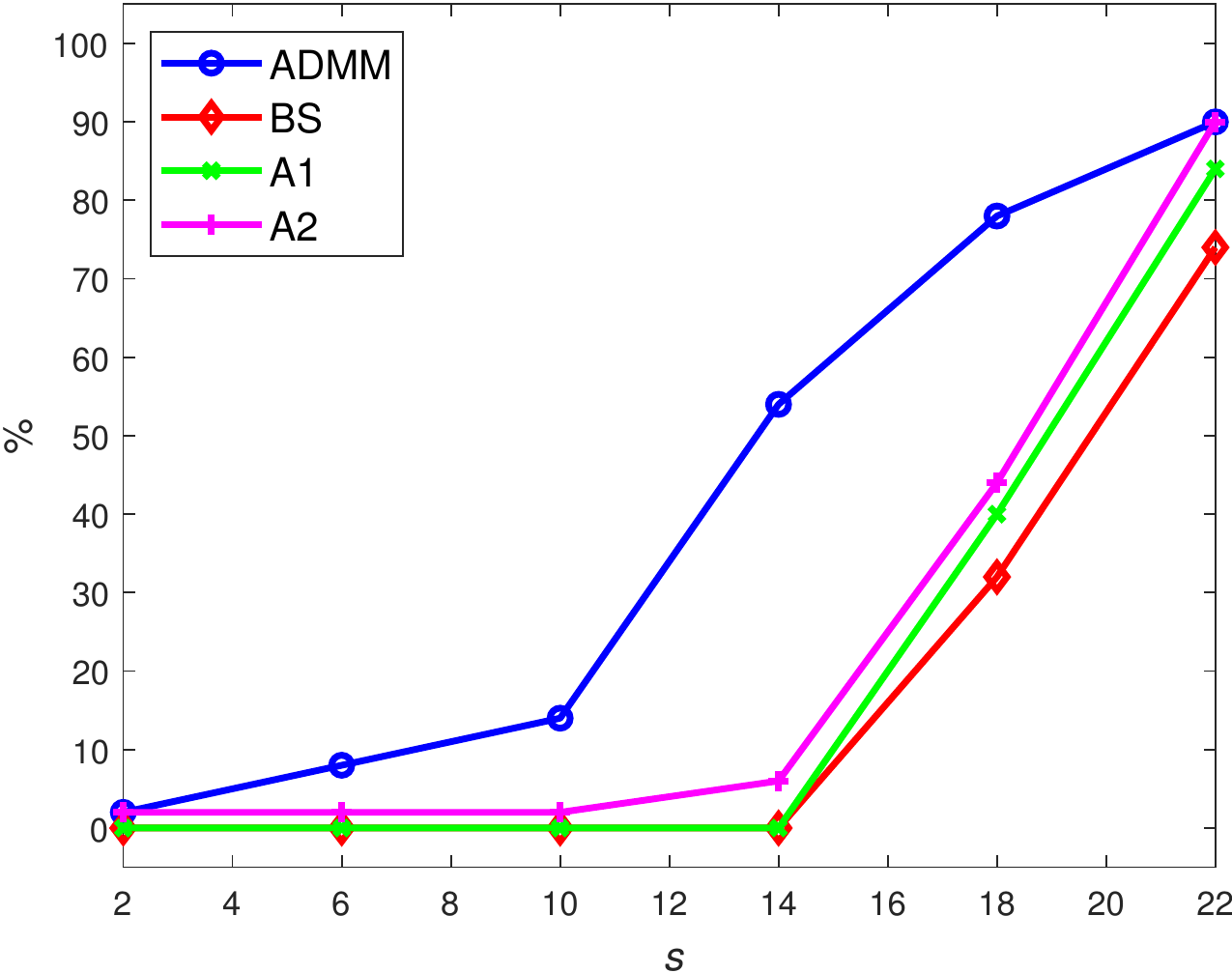}&
			\includegraphics[width=0.27\textwidth]{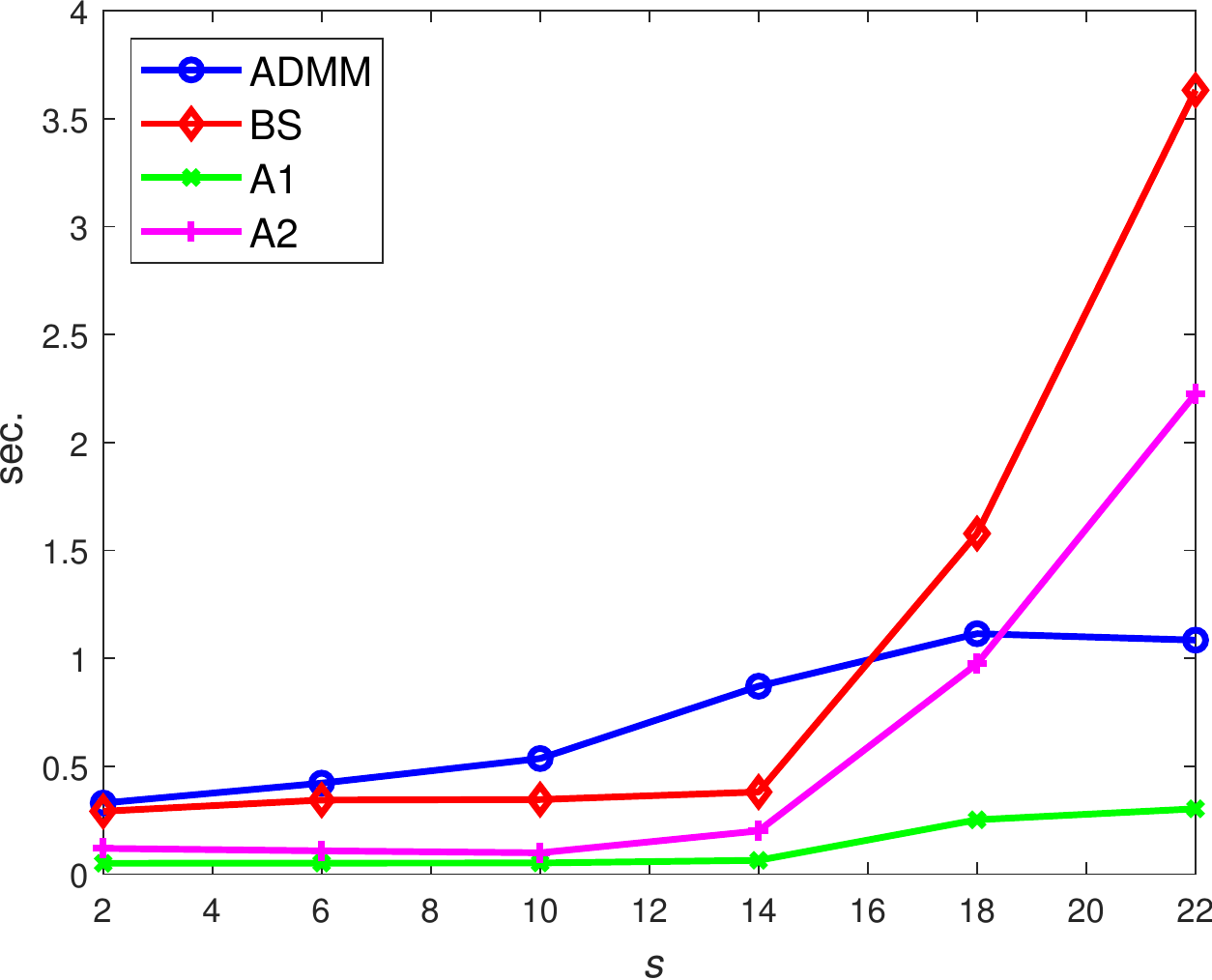}\\
		\end{tabular}
		\caption{Algorithmic comparison in the Gaussian distribution case with $F = 1$ (top) and $F = 20$ (bottom) in terms of success rates (left), algorithm-failure rates (middle), and computation time (right). }
		\label{fig:alg_SR_time_G}
	\end{figure*}


	\begin{figure}[t]
		\centering 
		\begin{tabular}{cc}
			$F= 1, D = 3$ & $F = 20, D=3$ \\		\includegraphics[width=0.2\textwidth]{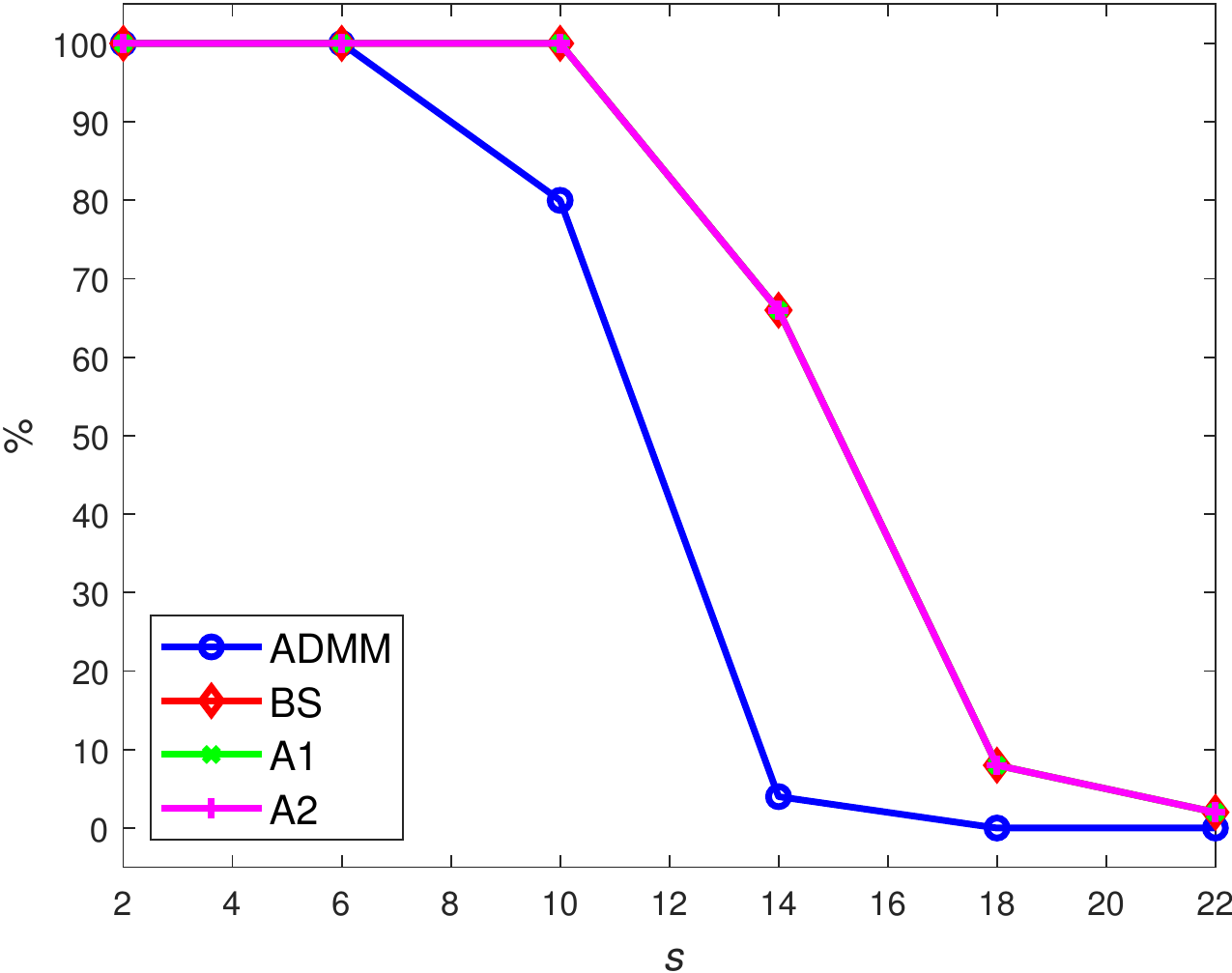} &
			\includegraphics[width=0.2\textwidth]{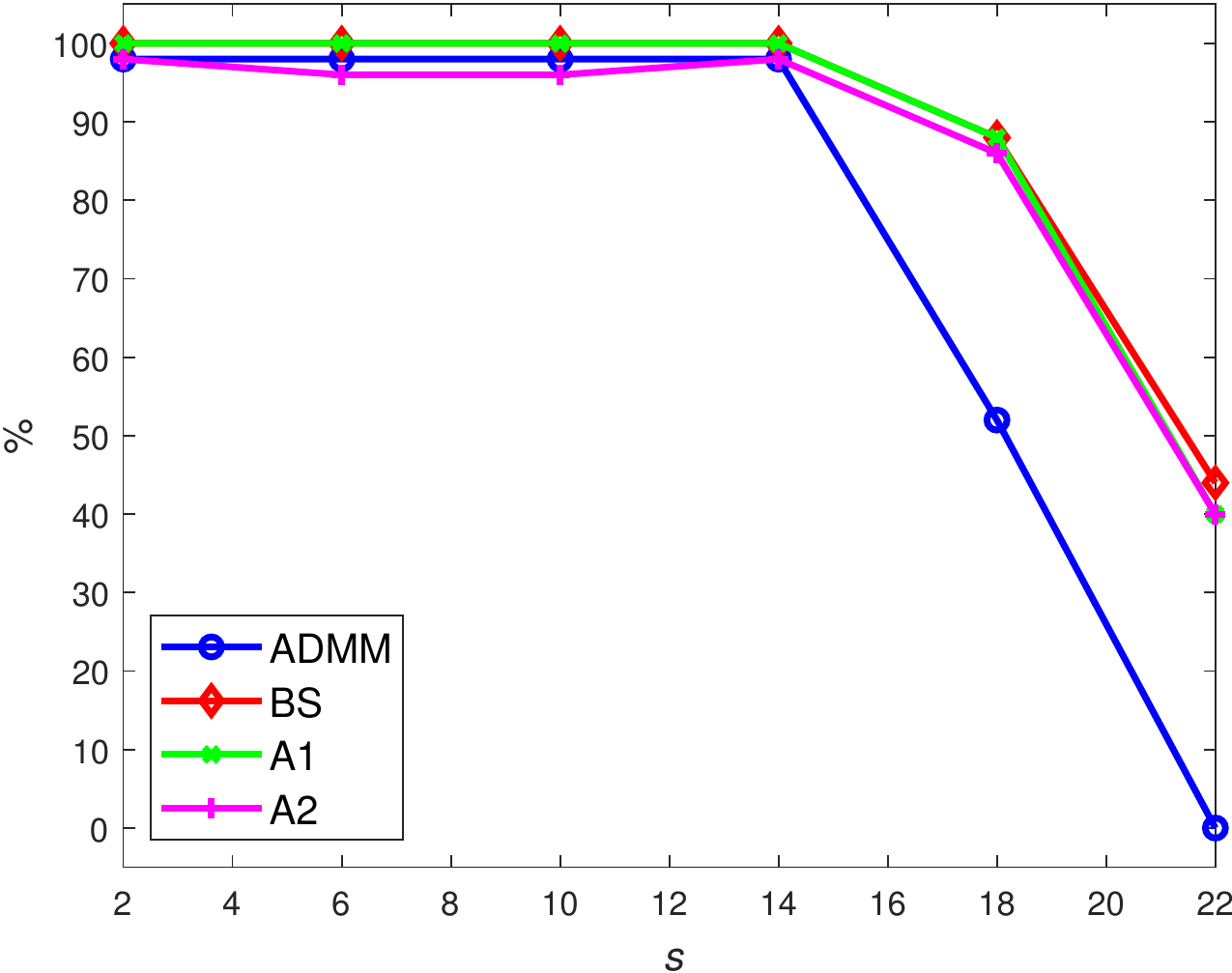}\\
			$F = 1, D=5$ & $F = 20, D=5$\\
			\includegraphics[width=0.2\textwidth]{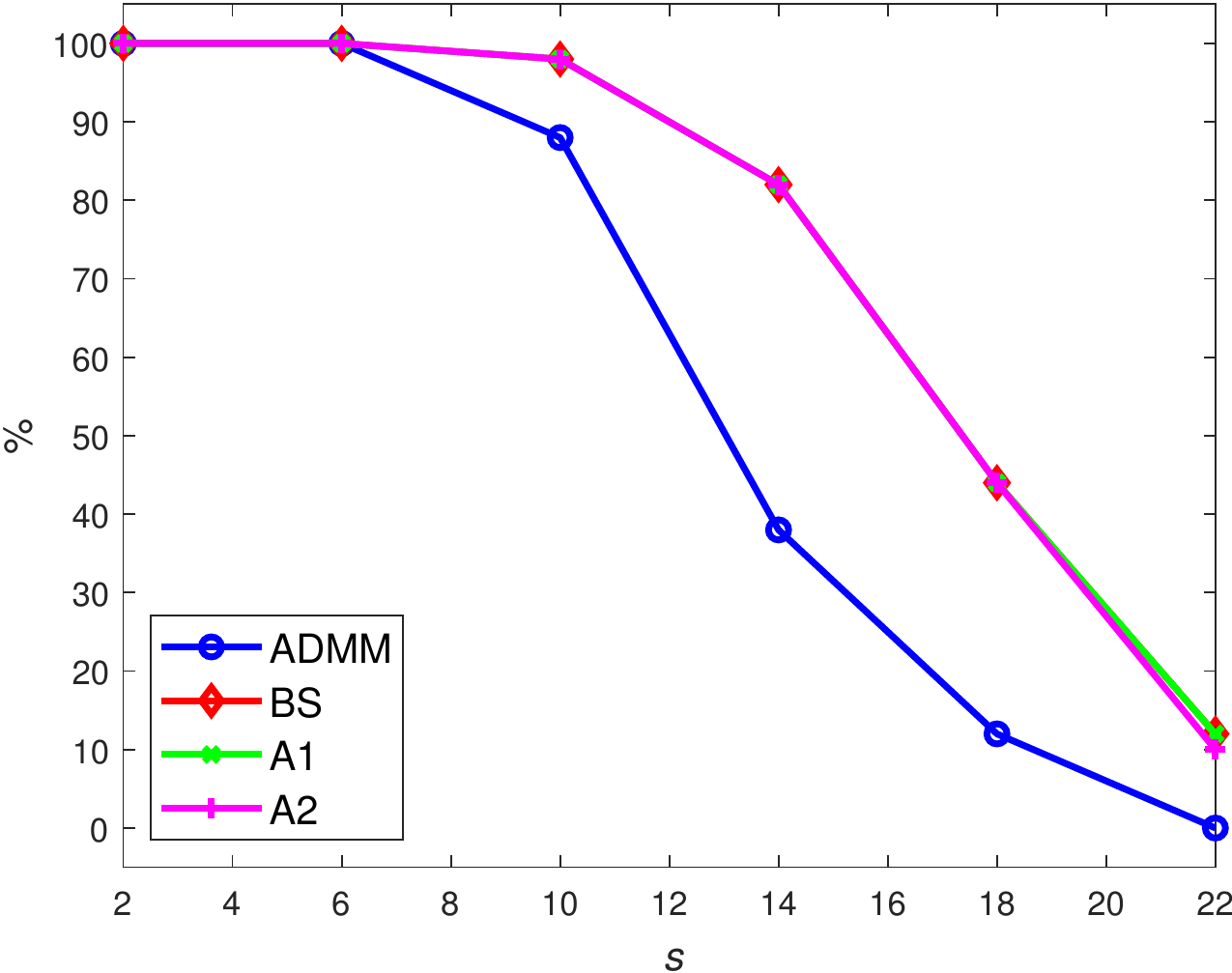}&
			\includegraphics[width=0.2\textwidth]{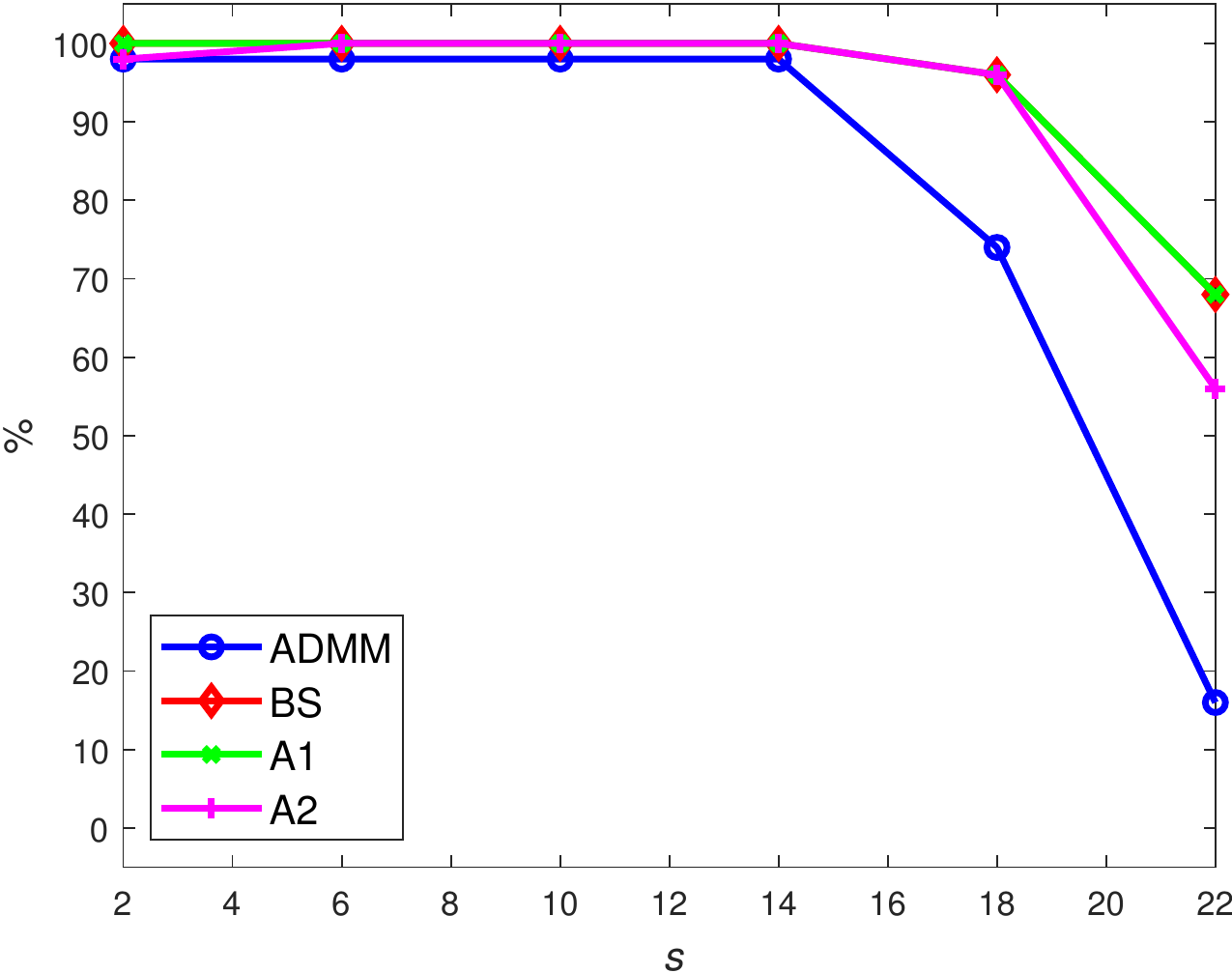}\\
		\end{tabular}
		\caption{Success rates of different $L_1/L_2$ minimizing  algorithms  versus sparsity at coherence levels $F=1$ (left) and $F=20$ (right) as well as  high dynamic ranges of $D=3$ (top) and $D=5$ (bottom). . }
		\label{fig:alg_SR_time_D35}
	\end{figure}

	\subsection{Model Comparison}
	
	We intend to compare  various sparse promoting models. Since the Gaussian case was conducted in our previous work \cite{l1dl2}, we focus on the dynamic range here. We  compare the proposed $L_1/L_2$ model with the following models:  $L_1$ \cite{chenDS98}, $L_p$ \cite{chartrand07},  $L_1$-$L_2$ \cite{yinLHX14,louYHX14}, and TL1 \cite{zhangX18}. We adopt 
	$L_1/L_2$-A1 to solve for the ratio model, as it is  the most efficient algorithm from the discussion in \Cref{sect:alg}.  The initial guess for all  non-convex models is the $L_1$ solution obtained by Gurobi. We choose  $p = 1/2$ for $L_p$ and $a = 10^{D-1}$ for TL1 when the range factor $D$ is known \textit{a priori}. 
	
	\Cref{fig:compare_all_F1_F20_min} plots the success rates of  $F = 1, 20$ and $D = 3,5$. We observe that TL1 is the best except for the low coherence and the low dynamic case, where $L_p$ is the best. But $L_p$ is the worst in the other cases. The $L_1/L_2$ model is always the second best. Note that the ratio model is parameter-free, while the performance of TL1 largely relies on the parameter $a$. \Cref{fig:compare_TL1_F1_F20_min} examines the success rate of TL1 with different values of $a$. We choose  $a = 10^{D-1}$ in the model comparison, which is almost the best among these testing values of $a$. If no such prior information of the dynamic range  were available to tune $a$, the performance of TL1 might be worse than $L_1/L_2$. 
	
	\begin{figure}[tbhp]
		\centering 
		\begin{tabular}{cc}
			$F = 1, D=3$ & $F = 20, D=3$ \\		\includegraphics[width=0.2\textwidth]{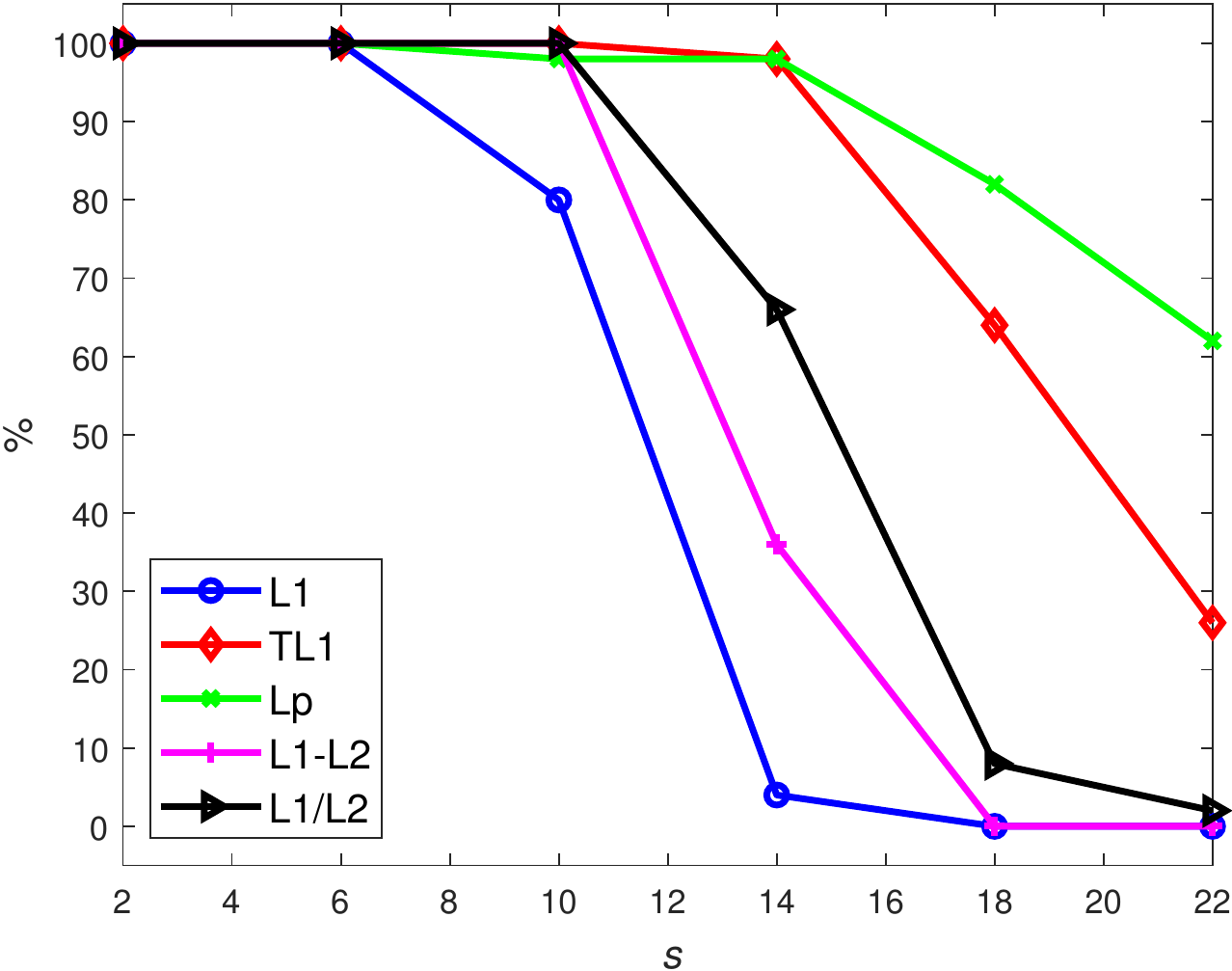} &
			\includegraphics[width=0.2\textwidth]{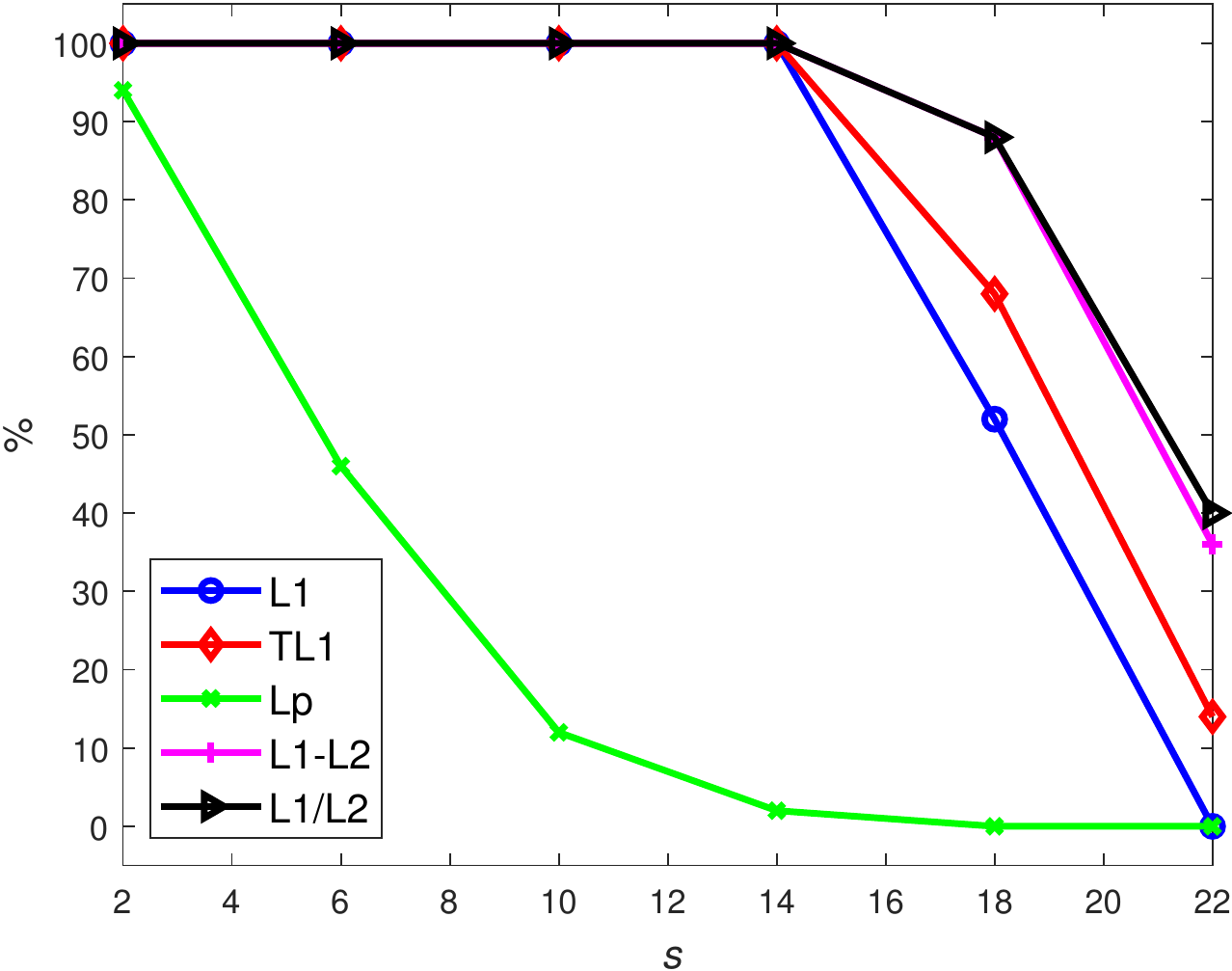}\\
			$F = 1, D=5$ & $F = 20, D=5$\\
			\includegraphics[width=0.2\textwidth]{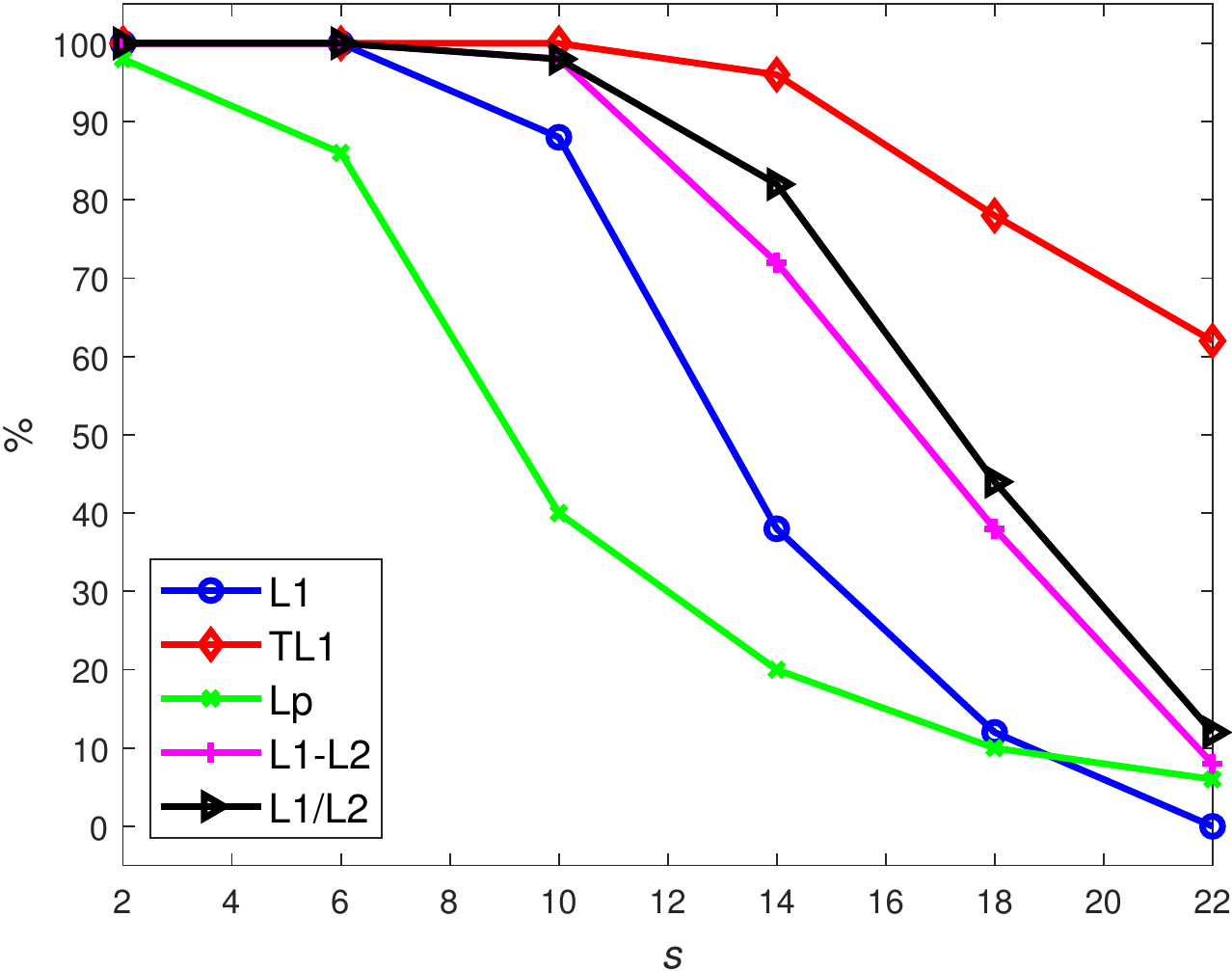}&
			\includegraphics[width=0.2\textwidth]{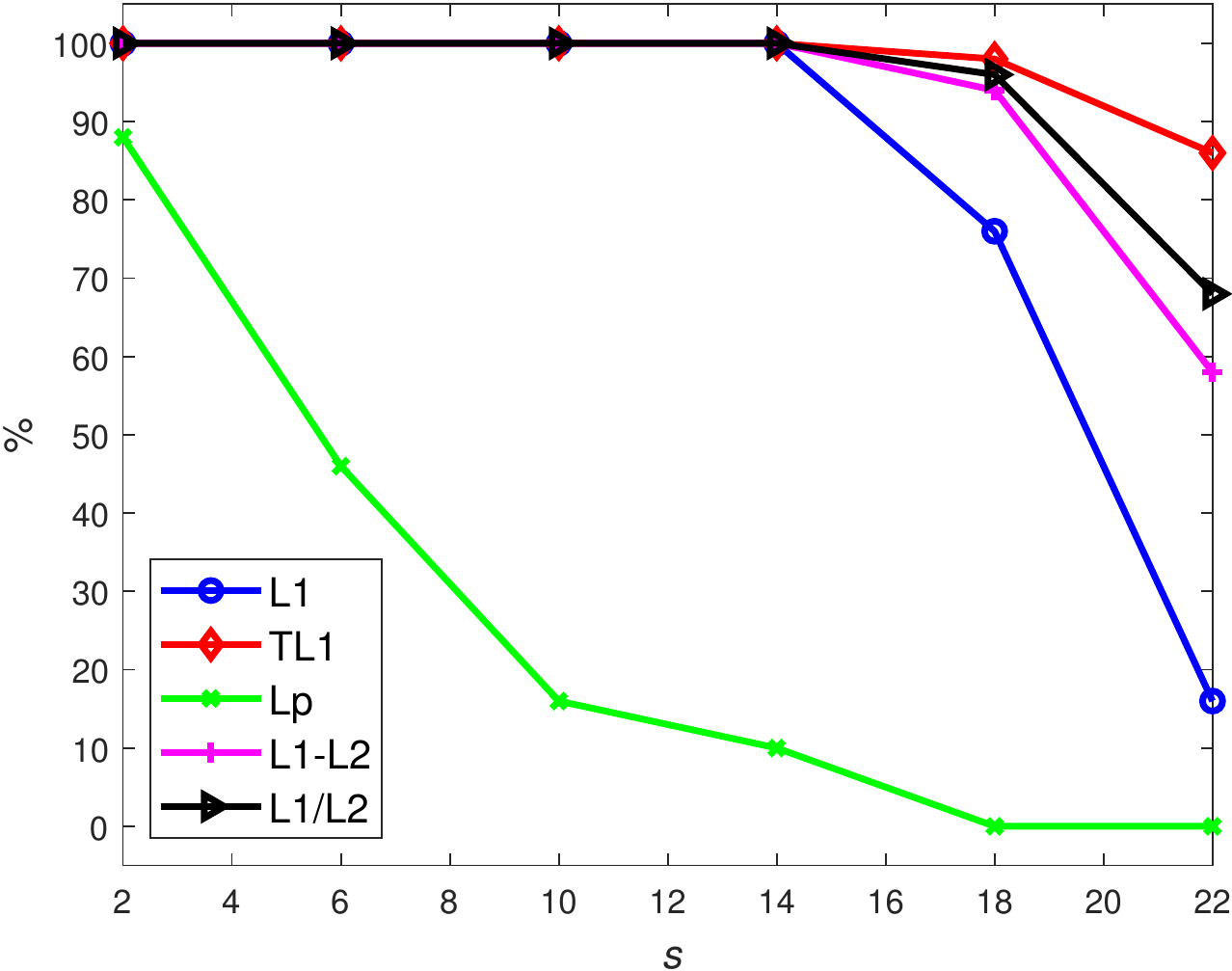}\\
		\end{tabular}
		\caption{Success rates of different models versus sparsity at coherence levels $F=1$ (left) and $F=20$ (right) as well as  high dynamic ranges of $D=3$ (top) and $D=5$ (bottom). }
		\label{fig:compare_all_F1_F20_min}
	\end{figure}
	
	\begin{figure}[tbhp]
		\centering 
		\begin{tabular}{cc}
			$F = 1, D=3$ & $F = 20, D=3$ \\
			\includegraphics[width=0.2\textwidth]{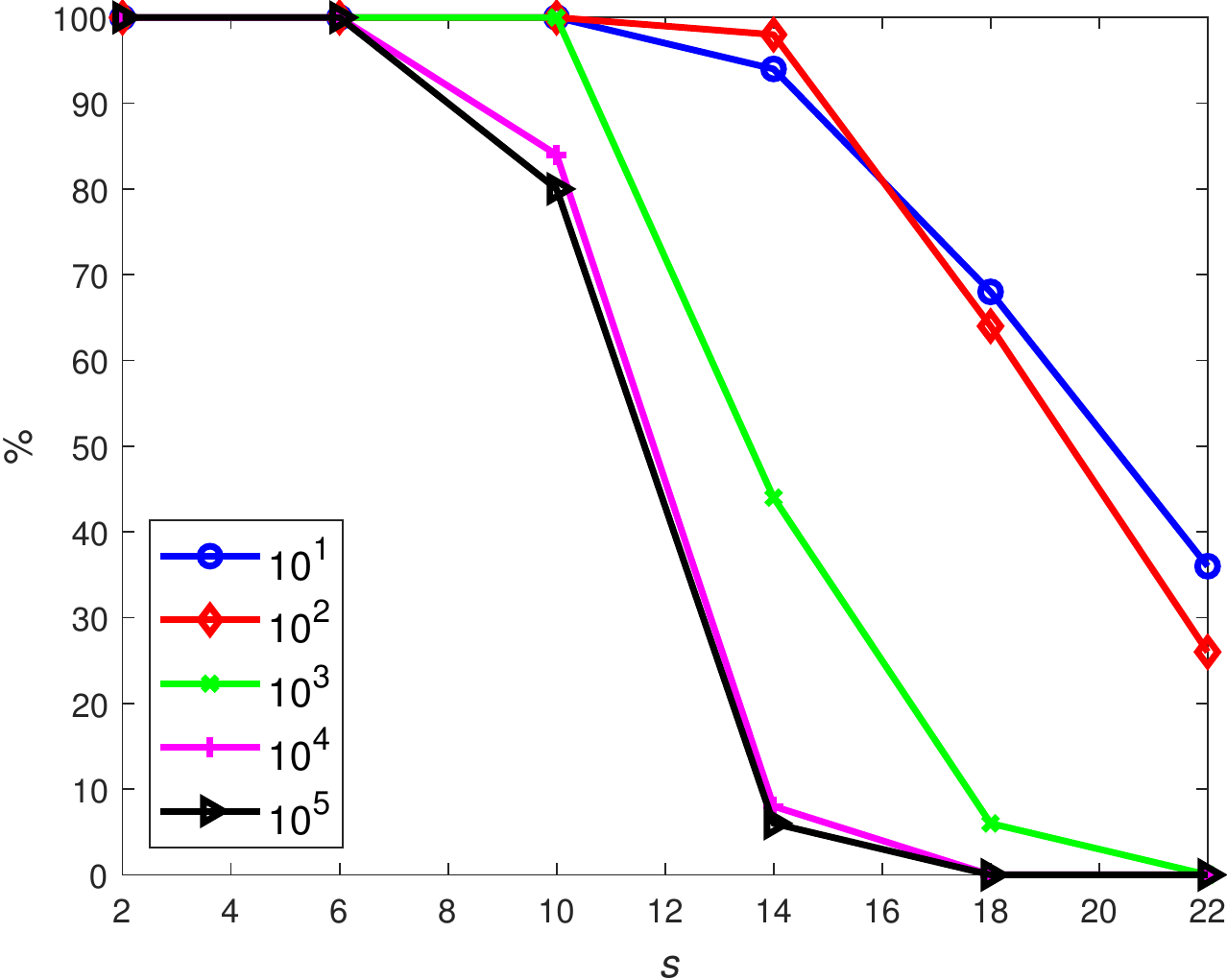}&
			\includegraphics[width=0.2\textwidth]{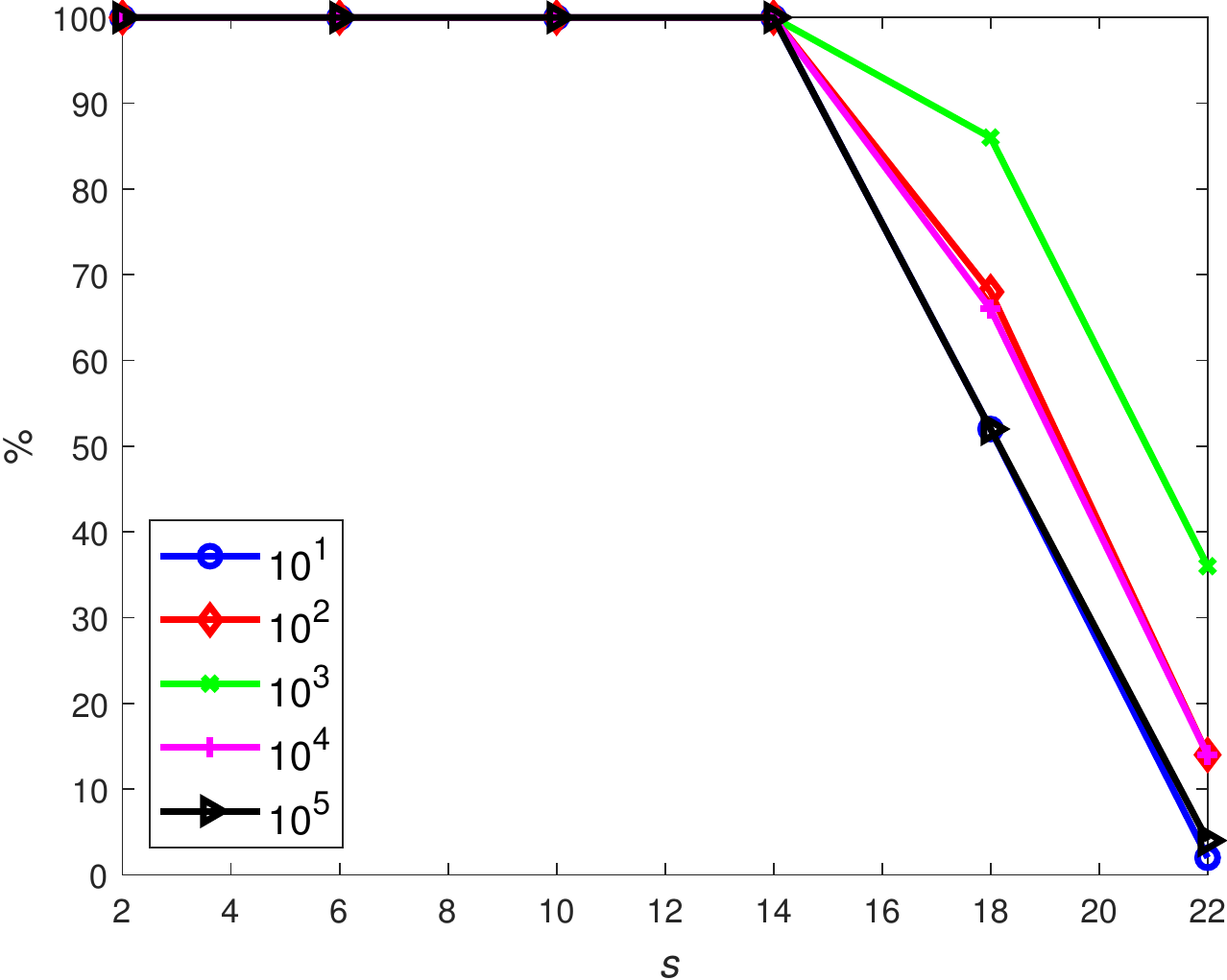}\\
			$F = 1, D=5$ & $F = 20, D=5$\\
			\includegraphics[width=0.2\textwidth]{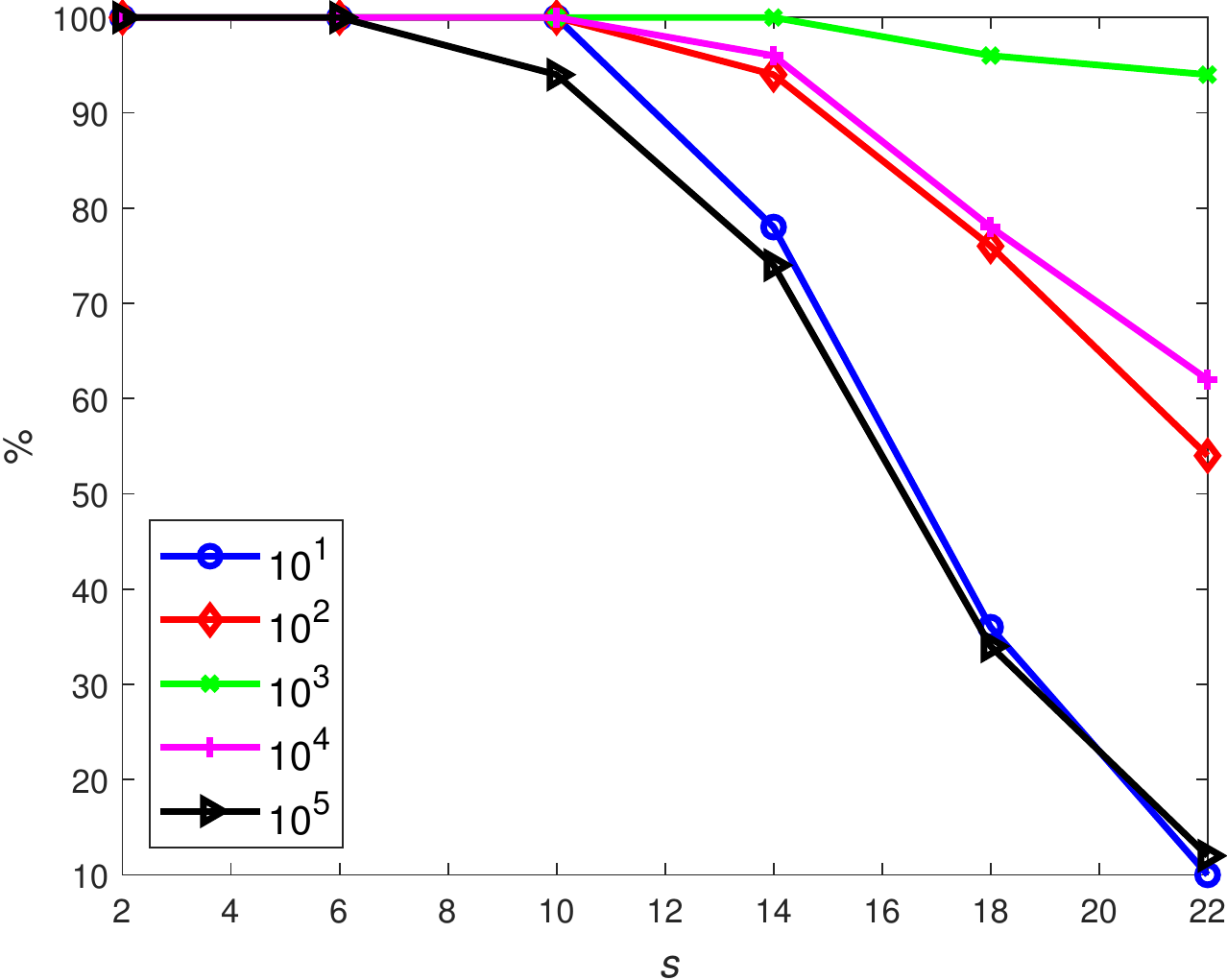}&
			\includegraphics[width=0.2\textwidth]{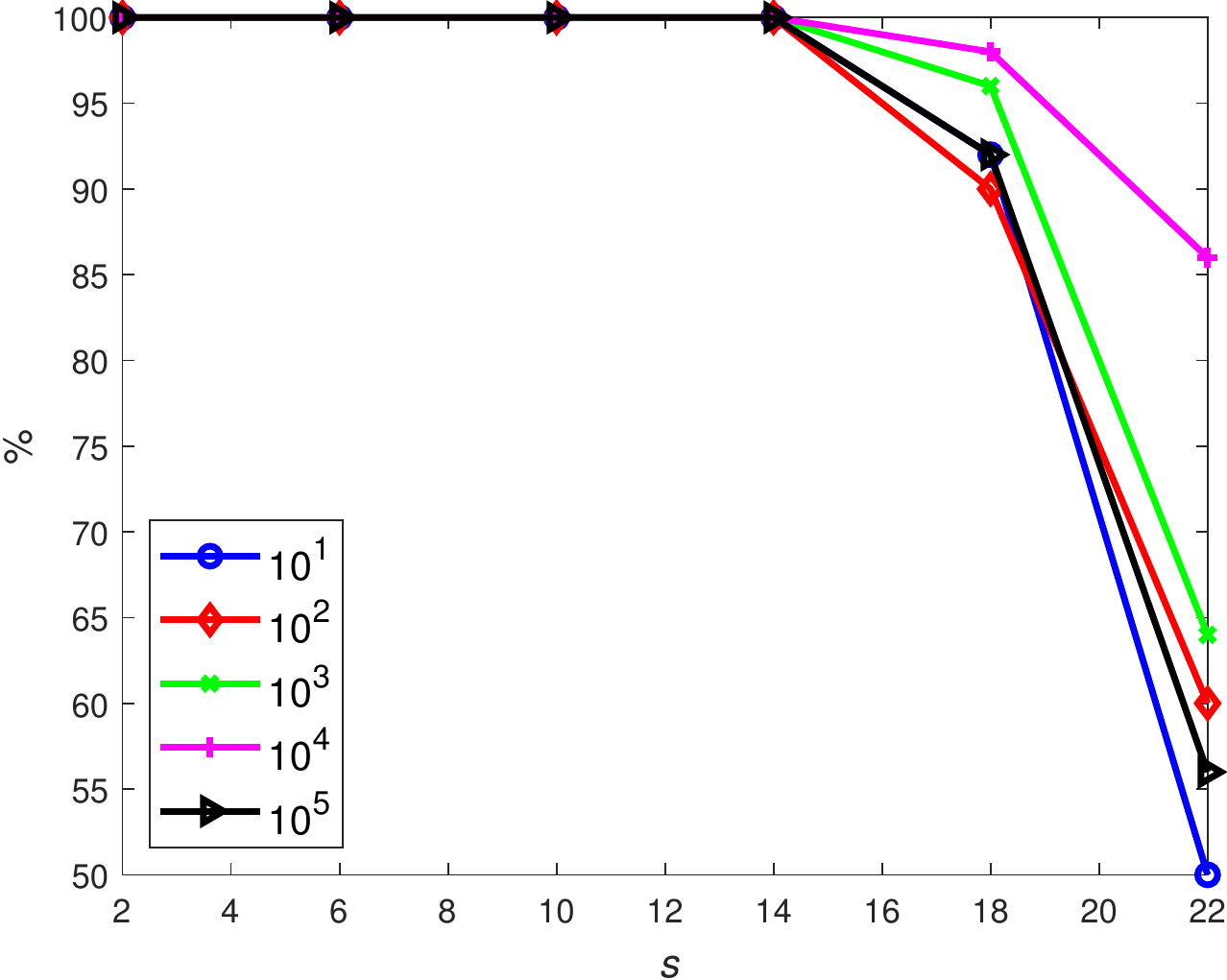}\\
		\end{tabular}
		\caption{Success rates of  different $a$   values  in  the TL1 model at coherence levels $F=1$ (left) and $F=20$ (right) as well as  high dynamic ranges of $D=3$ (top) and $D=5$ (bottom). }
		\label{fig:compare_TL1_F1_F20_min}
	\end{figure}

	\section{Discussions}\label{sect:discussion}
	Cand\'es and Wakin \cite{candes2008introduction} presented two principles in compressed sensing, i.e., \textit{sparsity} and \textit{incoherence}. 
	We reported in our previous work \cite{l1dl2} that higher coherence leads to better sparse recovery, which seems to contradict with the current belief in CS. 
	In this paper, we discuss  the dynamic range  and reveal its effect on the exact recovery via the $L_1$ approach. To our best of our knowledge, there has  been little discussion on the dynamic range in the CS literature, except for \cite{lorenz2011constructing}. We consider low-coherent matrices with $F=1$ and high-coherent ones with $F=20$. We record the success rates of different combinations of sparsity levels ($s=2:4:22$) and dynamic ranges $D=0:5$ in \Cref{tab:compare_dynamic_L1}. It shows that a higher dynamic range leads a better performance. It seems that the $L_1$ approach is independent on $D$ for relatively sparser signals.
	
	Now that there are three quantities that may contribute to the success of sparse recovery, i.e., sparsity, coherence, and dynamic range, we try to give a comprehensive analysis by using the relative error $\|\h x^\ast-\h x\|_2 / \|\h x\|_2$ instead of the success rates, as the latter depends on the successful threshold. 
	We plot in \Cref{fig:highdymanic} the mean and the standard deviation of the relative errors from 50 random trails versus  coherence levels ($F= 1, 5, 10, 15, 20$). Based on \Cref{tab:compare_dynamic_L1}, we only consider the number of non-zeros value larger than 18 and $D \geq 3$. In each subfigure of \Cref{fig:highdymanic}, the curves decrease when $F$ increases, which means that higher coherence leads to better performance. This is consistent with the observation in \cite{l1dl2}. As for the dynamic range, we discover in \Cref{fig:highdymanic} that a larger value of $D$ leads to a smaller relative error.
	Finally, the sparsity affects the performance in the way that smaller relative errors can be achieved for sparser signals.
	These numerical phenomena have not been reported in the CS literature, which motivate for future theoretical justifications.

	%
	%
	%
	%

	\begin{table}[t]
		\caption{Success rate (\%) in solving different dynamic ranges via  the $L_1$ model at two coherence levels $F=1$ and $F=20$. }\label{tab:compare_dynamic_L1} \begin{center}
			\footnotesize
			\begin{tabular}{c|c|c|c|c|c|c} \hline 
				\multicolumn{7}{c}{$F=1$} \\ \hline
				$s$ & 2  &   6  &  10  &  14  &  18  &  22  \\ \hline
				$D = 0$ & 100 & 100 & 80 & 4 & 0 & 0 \\ \hline
				$D = 1$ & 100 & 100 & 80 & 4 & 0 & 0 \\ \hline
				$D = 2$ & 100 & 100 & 80 & 4 & 0 & 0 \\ \hline
				$D = 3$ & 100 & 100 & 80 & 4 & 0 & 0 \\ \hline
				$D = 4$ & 100 & 100 & 86 & 16 & 0 & 0 \\ \hline
				$D = 5$ & 100 & 100 & 88 & 38 & 12 & 0 \\ \hline\hline
				\multicolumn{7}{c}{$F=20$} \\ \hline
				$s$ & 2  &   6  &  10  &  14  &  18  &  22  \\ \hline
				$D = 0$ & 100 & 100 & 100 & 100 & 50 & 0 \\ \hline
				$D = 1$ & 100 & 100 & 100 & 100 & 52 & 0 \\ \hline
				$D = 2$ & 100 & 100 & 100 & 100 & 52 & 0 \\ \hline
				$D = 3$ & 100 & 100 & 100 & 100 & 52 & 0 \\ \hline
				$D = 4$ & 100 & 100 & 100 & 100 & 54 & 0 \\ \hline
				$D = 5$ & 100 & 100 & 100 & 100 & 76 & 16 \\ \hline
			\end{tabular}

		\end{center}
		
	\end{table}
	
	
	\begin{figure*}[htpb]
		\centering 
		\begin{tabular}{ccc}
			$s=18$ & $s=20$ & $s=22$   \\
			\includegraphics[width=0.27\textwidth]{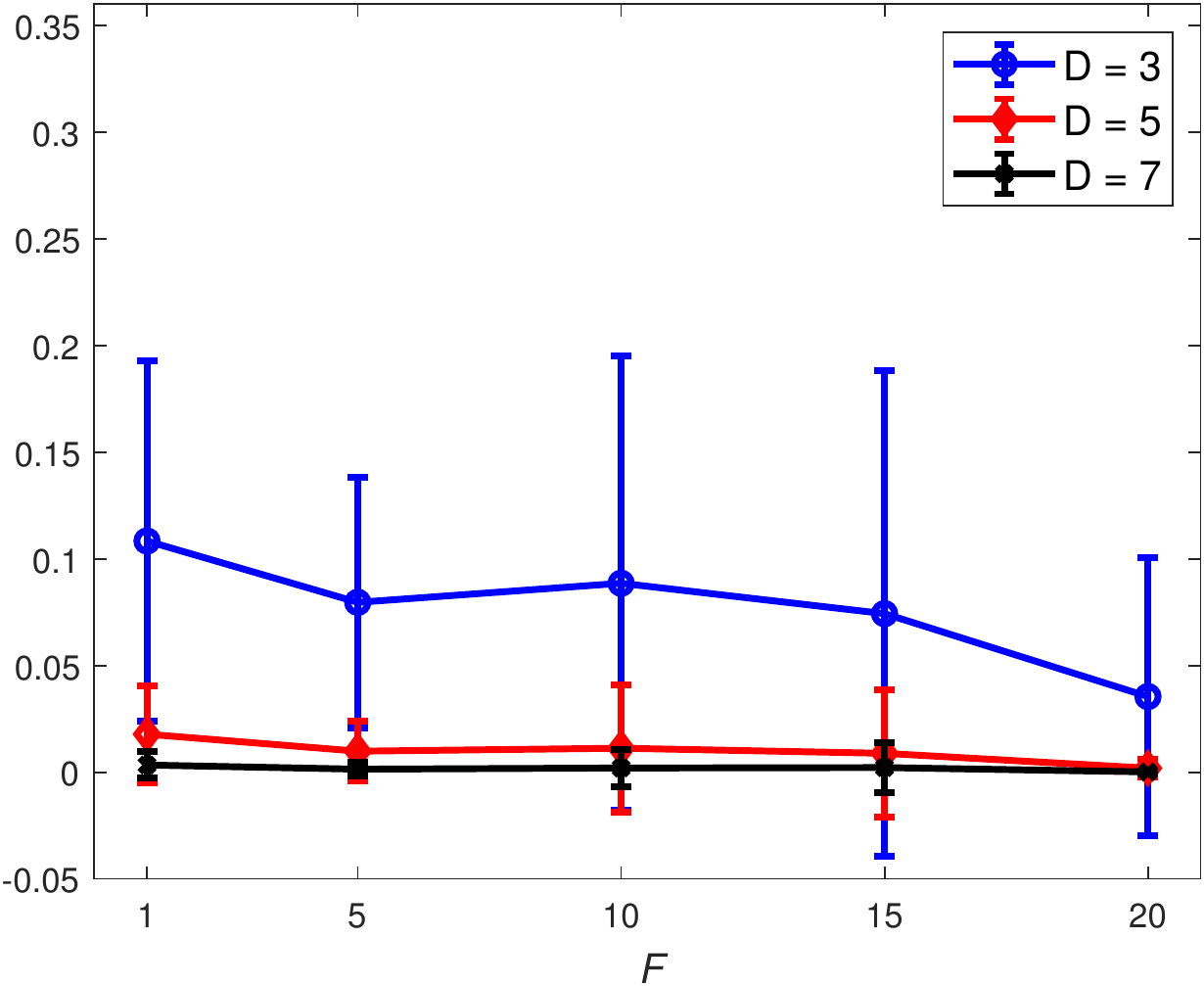}&
			\includegraphics[width=0.27\textwidth]{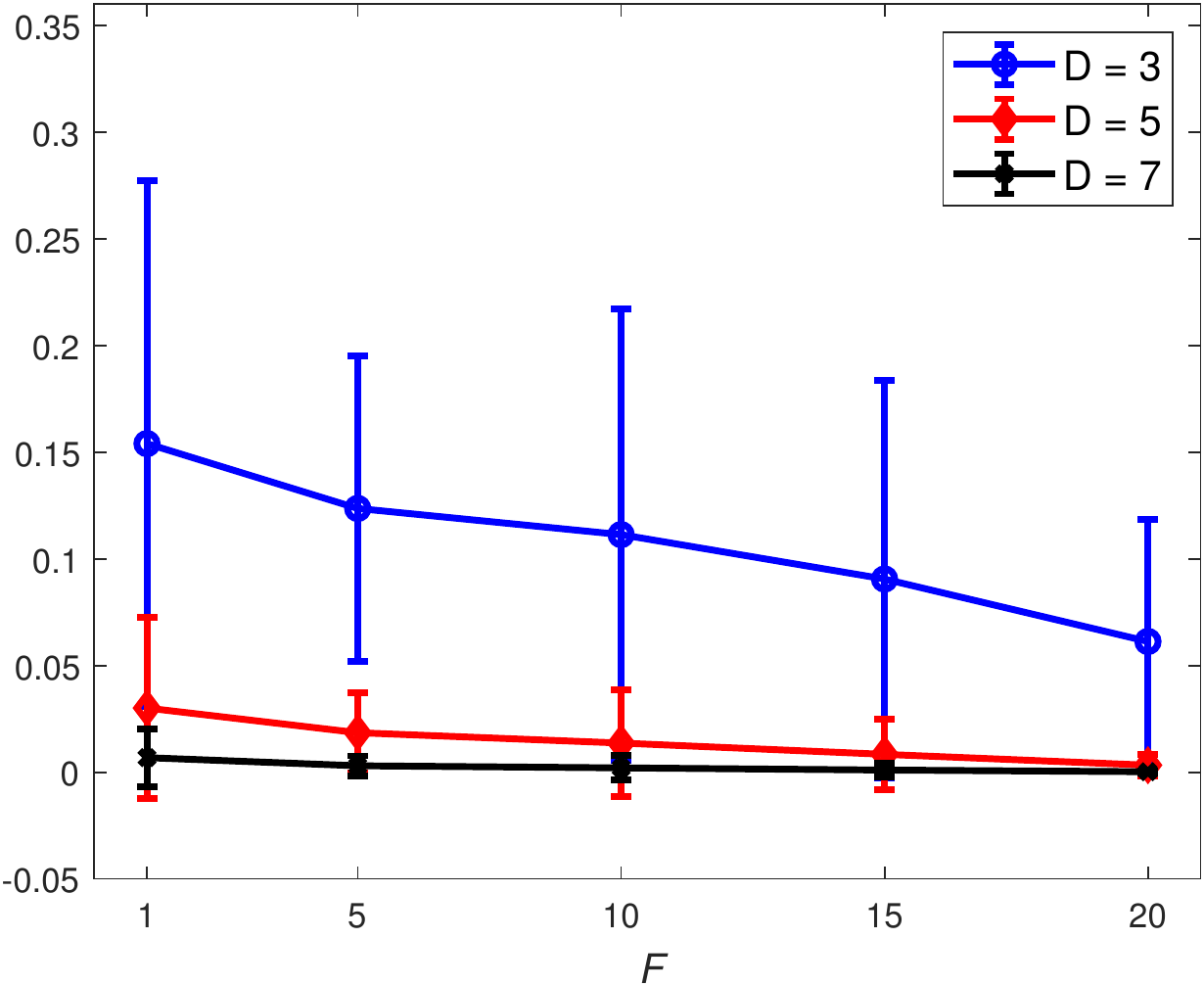}&
			\includegraphics[width=0.27\textwidth]{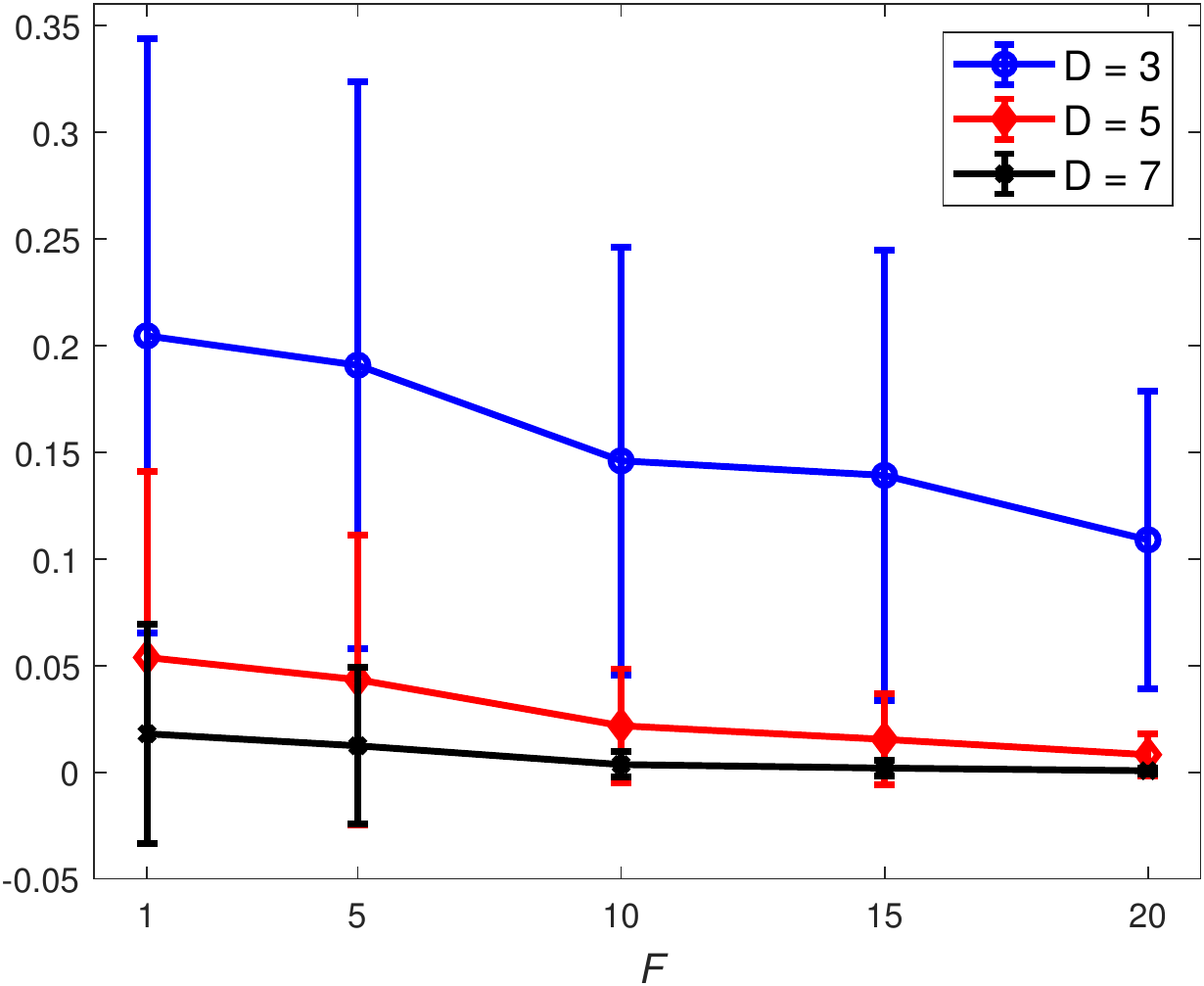}\\
		\end{tabular}
		\caption{The relative errors  $\|\h x^\ast-\hat{\h x}\|_2/\|\hat{\h x}\|_2$ produced by the $L_1$ approach for sparse signal recovery.  }
		\label{fig:highdymanic}
	\end{figure*}

	%
	%
	%

	\section{Conclusions and future works}\label{sect:conclusion}
	We studied the scale-invariant and parameter-free minimization $L_1/L_2$ to promote sparsity. We presented three numerical algorithms to minimize this nonconvex model based on the relationship between $L_1/L_2 $ and  $L_1$-$\alpha L_2 $ for certain $\alpha$. Experimental results compared the proposed algorithms with state-of-the-art methods in sparse recovery. Particularly important is the proposed algorithm works well  when the ground-truth signal has a high dynamic range. Last but not least, we analyzed the behaviors of the $L_1$ approach towards the exact recovery with varying sparsity, coherence, and dynamic range. Future works include the theoretical analysis on the effect of the high dynamic range towards sparse recovery as well as the applications of the ratio model in image processing such as blind deconvolution \cite{krishnan2011blind,repetti2015euclid}.

\appendix
\subsection{Proof of \Cref{lema:decreasing} }
\begin{proof}
		Based on the $\h x$-subproblem  in \eqref{equ:a2},
		we get
		\begin{equation*}
		\begin{split}
		&\left\|\h x^{(k+1)}\right\|_1  - \left\langle \h x^{(k+1)},  \frac{\alpha^{(k)}\h x^{(k)}}{\|\h x^{(k)}\|_2} \right\rangle + \frac{\beta}{2}\left\|\h x^{(k+1)} -\h x^{(k)}\right\|_2^2 \\  \leq &
		\left\|\h x^{(k)}\right\|_1  - \left\langle \h x^{(k)},  \frac{\alpha^{(k)}\h x^{(k)}}{\|\h x^{(k)}\|_2} \right\rangle + \frac{\beta}{2}\|\h x^{(k)} -\h x^{(k)}\|_2^2   \\  = & \left\|\h x^{(k)}\right\|_1  - \left\langle \h x^{(k)},  \frac{\alpha^{(k)}\h x^{(k)}}{\|\h x^{(k)}\|_2} \right\rangle.
		\end{split}
		\end{equation*} 
		After rearranging, we get the following inequality
		\begin{equation}\label{equ:l1_1}
		\begin{split}
		&	\|\h x^{(k+1)}\|_1  + \frac{\beta}{2}\|\h x^{(k+1)} -\h x^{(k)}\|_2^2    \\  \leq& \|\h x^{(k)}\|_1+ \alpha^{(k)}\left\langle \h x^{(k+1)}- \h x^{(k)},  \frac{\h x^{(k)}}{\|\h x^{(k)}\|_2} \right\rangle \\
		\leq  & \|\h x^{(k)}\|_1  + \alpha^{(k)}\left(\|\h x^{(k+1)}\|_2- \|\h x^{(k)}\|_2 \right)\\
		=& \alpha^{(k)}\|\h x^{(k+1)}\|_2.
		\end{split}
		\end{equation}	
		The second inequality is	owing to the convexity of Euclidean norm and the definition of $\alpha^{(k)}$. \Cref{lema:decreasing} is then obtained  by dividing $\|\h x^{(k+1)}\|_2$ on both sides of  \eqref{equ:l1_1}. 
	\end{proof}
	
	\subsection{Proof of \Cref{lema_lipschitz} }
	\begin{proof}
	Simple calculations lead to 
		\begin{equation}
		\label{equ:lip1}
		\begin{split}
		&\left\|\frac{\h x}{\|\h x\|_2} -  \frac{\h y}{\|\h y\|_2}\right\|_2^2  =  1 - \frac{ 2\langle\h x, \h y\rangle}{\|\h x \|_2 \| \h y\|_2} + 1\\
		= &  \frac{1}{\|\h x\|_2\|\h y\|_2}\Big(2\|\h x\|_2\|\h y\|_2 -2 \langle \h x, \h y \rangle\Big)\\
		\leq& \frac{1}{\|\h x\|_2\|\h y\|_2} \Big( \| \h x\|_2^2 +\| \h y\|_2^2 -2\langle \h x,\h y \rangle\Big)\\
		=&\frac{1}{\|\h x\|_2\|\h y\|_2} \|\h x - \h y\|_2^2.
		\end{split}
		\end{equation}
		For any $\h x$ satisfying $A \h x = \h b$, the minimal $L_2$ norm is reached by projecting the origin $\h 0$ onto the feasible set of $\{\h x\ |\ A \h x = \h b\}.$  It follows from the projection operator defined in \eqref{equ:projection} that
		\begin{equation}
		\label{equ:lip2}
		\|\h x\|_2  \geq  \|\mathbf{proj}(\h 0)\|_2 =  \| A^T(A A^T)^{-1}\h b\|_2.
		\end{equation}
		Combining \eqref{equ:lip1} and \eqref{equ:lip2}, we get \Cref{lema_lipschitz}. 
	\end{proof}
	
	\subsection{Proof of \Cref{F_lipschitz}}
	\begin{proof}
		It is straightforward to have
		\begin{equation}\label{eq:lemma3_all}
		\begin{split}
		&\left\|\nabla w(\h x)- \nabla w(\h y)\right\|_2   = \left\|\frac{\| \h x \|_1}{\| \h x \|_2^2}\h x - \frac{\| \h y\|_1}{\| \h y \|_2^2}\h y\right\|_2 \\
		=& \left\|\frac{\| \h x \|_1}{\| \h x \|_2^2}\h x -\frac{\| \h x \|_1}{\| \h y\|_2^2}\h y + \frac{\| \h x \|_1}{\| \h y \|_2^2}\h y - \frac{\| \h y\|_1}{\| \h y \|_2^2}\h y\right\|_2 \\
		\leq &\|\h x \|_1  \left\| \frac{\h x}{\| \h x \|_2^2} -\frac{\h y}{\| \h y\|_2^2}\right\|_2 + \frac{1}{\|\h y\|_2}  \Big|\|\h x\|_1 - \|\h y\|_1 \Big|.\\
		\end{split}
		\end{equation}
	We simplify  the first term in \eqref{eq:lemma3_all} by calculating
		\begin{equation*}
		\begin{split}
		&	\left\| \frac{\h x}{\| \h x \|_2^2} -\frac{\h y}{\| \h y\|_2^2}\right\|_2^2  =  \frac{1}{\|\h x\|_2^2}+ \frac{1}{\|\h y\|_2^2}-\frac{2\langle\h x, \h y\rangle}{\|\h x\|_2^2\|\h y\|_2^2}\\
		=&  \frac{\|\h x\|_2^2+\|\h y\|_2^2 - 2\langle\h x, \h y\rangle}{\|\h x\|_2^2\|\h y\|_2^2}
		= \left(\frac{\|\h x-\h y\|_2}{\|\h x\|_2\|\h y\|_2}\right)^2,
		\end{split}
		\end{equation*}
	and	using $\|\h x \|_1\leq \sqrt{n}\| \h x\|_2$. Therefore, we get 
		\begin{equation}\label{equ:firstterm}
			\|\h x\|_1\left\| \frac{\h x}{\| \h x \|_2^2} -\frac{\h y}{\| \h y\|_2^2}\right\|_2 \leq \sqrt{n}L \| \h x - \h y\|_2.
		\end{equation}
	As for the second term in \eqref{eq:lemma3_all}, we have it bounded by
		\begin{equation}\label{equ:2ndterm}
		\begin{split}
			& \frac{1}{\|\h y\|_2} \Big|\|\h x\|_1 - \|\h y\|_1 \Big|  \leq \frac{1}{\|\h y\|_2}  \|\h x - \h y\|_1 \\
			&  \leq \frac{\sqrt{n}}{\|\h y\|_2}  \|\h x - \h y\|_2  \leq \sqrt{n} L  \|\h x - \h y\|_2.
		\end{split}
		\end{equation}
		Combining \eqref{equ:firstterm} and \eqref{equ:2ndterm}, we obtain 
		\eqref{eq:lemma3}.
	\end{proof}
	
	\subsection{Proof of \Cref{lemma:G}}
	\begin{proof} 
		It is straightforward that
		$$\Phi(\h x^\ast) = \h 0  \iff \h x^\ast = \mathrm{prox}_{\frac{1}{\beta}g}\left(\h x^{\ast} - \frac{1}{\beta}\nabla w(\h x^\ast)\right).$$ By the optimality condition~\cite{combettes2011proximal}, the latter relation holds if and only if there exists a vector $\h s $ such that
		\begin{equation*}
		\begin{split}
		0 & \in  \partial \| \h x^\ast \|_1 + \nabla w(\h x^\ast) + \beta\left(\h x^{\ast} - \h x^{\ast}  \right)  + A^T \h s \\
		& = \partial \| \h x^\ast \|_1 + \nabla w(\h x^\ast)   + A^T \h s,
		\end{split}
		\end{equation*}
		which implies that $\h x^*$ is a critical point of \eqref{equ:reformulation}. It follows from \eqref{eq:G} that \eqref{equ:reformulation} is equivalent to \eqref{equ:ratio} and hence $\h x^*$ is also a critical point of \eqref{equ:ratio}.
		According to the nonexpansiveness of the proximal operator and the Lipschitz continuousness of $\nabla w$, we have 
		\begin{equation*}
		\begin{split}
		& \left\|\Phi(\h x)-\Phi(\h y) \right\|_2  \\
		\leq &  \beta \left\|\mathrm{prox}_{\frac{1}{\beta}g}\left(\h x - \frac{1}{\beta}\nabla w(\h x)\right)- \mathrm{prox}_{\frac{1}{\beta}g}\left(\h y - \frac{1}{\beta}\nabla w(\h y)\right) \right\|_2\\
		& +  \beta \|\h x- \h y\|_2\\
		\leq &  \beta\left\|\left(\h x - \frac{1}{\beta}\nabla w(\h x)\right)-\left(\h y - \frac{1}{\beta}\nabla w(\h y)\right) \right\|_2 + \beta \|\h x- \h y\|_2\\
		\leq &  \|\nabla w(\h x) - \nabla w(\h y)\|_2 +  2\beta \|\h x- \h y\|_2\\
		\leq & (L_w+2\beta) \|\h x- \h y\|_2.  
		\end{split}
		\end{equation*}	
		The Lemma follows.
	\end{proof}
	
	
	%
	
	%
	%
	%
	
	%
	%


	
	
	%
	\bibliographystyle{IEEEtran}
	\bibliography{refer_l1dl2}

\end{document}